\documentclass[11pt,reqno]{amsart}

\usepackage[T1]{fontenc}
\usepackage[utf8]{inputenc}
\usepackage{microtype}
\usepackage{amsmath,amssymb,amsthm,mathtools}
\usepackage{enumitem}
\usepackage{booktabs}
\usepackage{mathrsfs}
\usepackage{xcolor}
\usepackage[colorlinks=true,linkcolor=blue!55!black,citecolor=green!45!black,urlcolor=blue!60!black]{hyperref}
\usepackage[nameinlink,capitalize,noabbrev]{cleveref}

\allowdisplaybreaks
\numberwithin{equation}{section}

\newtheorem{theorem}{Theorem}[section]
\newtheorem{proposition}[theorem]{Proposition}
\newtheorem{lemma}[theorem]{Lemma}
\newtheorem{corollary}[theorem]{Corollary}

\theoremstyle{definition}
\newtheorem{definition}[theorem]{Definition}

\theoremstyle{remark}
\newtheorem{remark}[theorem]{Remark}

\newcommand{\R}{\mathbb R}
\newcommand{\N}{\mathbb N}
\newcommand{\Sph}{\mathbb S}
\newcommand{\eps}{\varepsilon}
\newcommand{\dd}{\,\mathrm d}
\newcommand{\supp}{\operatorname{supp}}

\newcommand{\diver}{\operatorname{div}}
\newcommand{\tr}{\operatorname{tr}}

\newcommand{\sgn}{\operatorname{sgn}}

\newcommand{\cH}{\mathcal H}

\newcommand{\loc}{\mathrm{loc}}

\makeatletter
\newcommand{\nosection}[1]{%
  \begingroup
  \let\@tocwrite\@gobbletwo
  \section*{#1}%
  \endgroup
}
\makeatother

\title[Anisotropic Allen--Cahn transitions]{Entire monotone solutions of the anisotropic Allen--Cahn equation in dimension 5}

\author[H. Lu]{Haowen Lu}
\address{\noindent  Department of Mathematics, Chinese University of Hong Kong, Shatin, NT, Hong Kong}
\email{hwlu@math.cuhk.edu.hk}

\author[S. Wang]{Song Wang}
\address{\noindent  Department of Mathematics, Chinese University of Hong Kong, Shatin, NT, Hong Kong}
\email{songw@link.cuhk.edu.hk}

\author[J. Wei]{Juncheng Wei}
\address{\noindent Department of Mathematics, Chinese University of Hong Kong,
Shatin, NT, Hong Kong}
\email{wei@math.cuhk.edu.hk}

\author[Y. Wu]{Yuanze Wu}
\address{\noindent  School of Mathematics, Yunnan Normal University, Kunming, 650500, P.R. China;
Yunnan Key Laboratory of Modern Analytical Mathematics and Applications, Kunming, 650500, P.R. China}
\email{yuanze.wu@ynnu.edu.cn}

\date{August 2026}



\begin{document}

\begin{abstract}
In this paper, we consider the anisotropic Allen-Cahn equation
\begin{eqnarray*}
-\diver a(Du)+W'(u)=0\quad\text{in }\mathbb{R}^N,
\end{eqnarray*}
where $a(p):=DH(p)$ with $H(p)=\frac{1}{2}F(p)^2$ and $F$ a uniformly elliptic integrand, and $W(u)=\frac{1}{4}(1-u^2)^2$. Based on the Mooney-Yang anisotropic minimal graph, we prove that  the anisotropic Allen-Cahn equation admits a {\bf stable} solution for $N\geq4$ in the weak sense, {\bf whose level sets are not hyperplanes}.  As a byproduct, we also construct a smooth solution of the above anisotropic Allen-Cahn equation for $N\geq5$ that is {\bf monotone in one direction but is not one-dimensional}.
\end{abstract}

\maketitle
\tableofcontents

\section{Introduction}
The Allen--Cahn equation
\begin{equation}\label{Allen-Cahn}
\Delta u+u-u^3=0 \qquad \text{in }\mathbb{R}^N
\end{equation}
is a basic model for phase transitions in chemistry and physics. After introducing a small parameter
$\varepsilon>0$, one obtains the equivalent scaled equation
\begin{equation}\label{Allen-Cahn-2}
\varepsilon^2\Delta u+u-u^3=0 \qquad \text{in }\mathbb{R}^N,
\end{equation}
which is the Euler--Lagrange equation of the scaled Allen--Cahn energy
$$ \int \frac{1}{2} \varepsilon |\nabla u|^2+\frac{1}{4 \varepsilon} (1-u^2)^2. $$ The classical theory of Modica and Mortola
\cite{ModicaMortola1977,Modica1978,Modica1987} shows that, as
$\varepsilon\to0$, transition layers of bounded-energy families are
closely related to minimal hypersurfaces. This connection naturally
links the Bernstein problem for minimal hypersurfaces with rigidity
questions for entire solutions of the Allen--Cahn equation.

\subsection{The isotropic setting}

Motivated by the Bernstein problem, De Giorgi conjectured in 1978 that a
bounded entire solution of \eqref{Allen-Cahn} which is monotone in one
direction must be one-dimensional, at least for $N\leq8$. The conjecture was
proved by Ghoussoub and Gui \cite{GhoussoubGui1998} for $N=2$ and by
Ambrosio and Cabr\'e \cite{AmbrosioCabre2000} for $N=3$. For
$4\leq N\leq8$, Savin \cite{Savin2009} proved the conjecture under the
additional end-state assumption
\[
\lim_{x_N\to\pm\infty}u(x',x_N)=\pm1.
\]
Related rigidity results under uniform end-state assumptions were obtained in
\cite{BarlowBassGui2000,BerestyckiHamelMonneau2000,
CaffarelliCordoba2006,Farina1999}. On the other hand, del Pino,
Kowalczyk and Wei \cite{delPinoKowalczykWei2011} constructed a
counterexample in dimension $N=9$. Their construction starts from the
Bombieri--De Giorgi--Giusti minimal graph
\cite{BombieriDeGiorgiGiusti1969} and uses an infinite-dimensional
Lyapunov--Schmidt reduction.

There are also nonflat entire solutions closely related to minimal cones.
Cabr\'e and Terra \cite{CabreTerra2009} constructed saddle-shaped solutions
in every even dimension, with nodal set equal to the Simons cone, and
Cabr\'e \cite{Cabre2012} studied their uniqueness and stability.
Since monotone solutions are stable, one may also ask whether bounded stable
solutions must be one-dimensional. Ghoussoub and Gui \cite{GhoussoubGui1998}, and
Ambrosio and Cabr\'e \cite{AmbrosioCabre2000}    proved a
low-dimensional result $N=2$, while Pacard and Wei \cite{PacardWei2013}
constructed non-one-dimensional stable solutions in dimensions $N\geq8$.
Their construction starts from the Simons cone and uses an
infinite-dimensional Lyapunov--Schmidt reduction. The
$O(4)\times O(4)$ symmetry is important in the analysis of the Jacobi
operator and the reduced problem. Liu, Wang and Wei
\cite{LiuWangWei2017} later constructed global minimizers by a variational
argument combined with the Jerison--Monneau lifting
\cite{JerisonMonneau2004}.

On the other hand, the corresponding  limiting Bernstein problem for stable minimal hypersurfaces has
also received much attention. Classical results include
\cite{SchoenSimonYau75,SchoenSimon1981}, and recent progress can be found in
\cite{ChodoshLi24,Bellettini25,ChodoshLiMinterStryker2026,Mazet24}.

There are also many results relating Allen--Cahn solutions to minimal hypersurfaces. Hutchinson and Tonegawa \cite{HutchinsonTonegawa2000} proved that bounded-energy critical points converge, in the varifold sense, to stationary integral interfaces. Tonegawa and Wickramasekera \cite{TonegawaWickramasekera2012} developed
the corresponding regularity theory for stable interfaces. Guaraco \cite{Guaraco2018} used the Allen--Cahn functional to construct embedded minimal hypersurfaces by min--max methods, followed by further variational results of Gaspar and Guaraco
\cite{GasparGuaraco2018,GasparGuaraco2019}. Fine properties of stable and min--max transition layers were studied by Wang and Wei \cite{WangWei2019}, Mantoulidis \cite{Mantoulidis2021}, and Chodosh and Mantoulidis \cite{ChodoshMantoulidis2020}. Dey \cite{Dey2022} proved that
the Allen--Cahn and Almgren--Pitts widths agree. For limit interfaces with free boundary, see Li, Parise and Sarnataro
\cite{LiPariseSarnataro2024}.

\subsection{The anisotropic setting}
Recently, significant attention has been focused on the anisotropic Allen–Cahn equation. More precisely, let
$F:\mathbb{R}^N\to[0,\infty)$ be a uniformly elliptic, positively
one-homogeneous integrand, and set
\begin{equation}\label{eq:H-a-def}
    H(p):=\frac12F(p)^2, \qquad a(p):=DH(p).
\end{equation}
Then the anisotropic Allen--Cahn equation is
\begin{equation}\label{eq:main-PDE}
-\varepsilon^2\operatorname{div} a(Du)+W'(u)=0
\qquad\text{in }\mathbb{R}^N,
\end{equation}
and its associated energy is
\begin{equation}\label{eq:energy}
E_\varepsilon(u;\Omega)
:=
\int_\Omega\left[\varepsilon H(Du)+\frac1\varepsilon W(u)\right]dx,
\qquad
W(u):=\frac14(1-u^2)^2.
\end{equation}
The classical $\Gamma$-convergence of \eqref{eq:energy} to anisotropic
perimeter goes back to Bouchitt\'e \cite{Bouchitte1990}. More recently,
De Rosa and Pigati \cite{DeRosaPigati2025} developed an anisotropic
Allen--Cahn min--max theory and a detailed connection between diffuse
critical points and anisotropic minimal hypersurfaces for the surface
functional
\[
\mathcal A_F(M)
=
\int_M F(\nu_M)\,\dd\mathcal H^{N-1}.
\]

The existence and regularity theory for anisotropic minimal surfaces has
a long history, starting with the work of Almgren \cite{Almgren1968} and
Allard \cite{Allard1983}. More recent developments include the work of
De Philippis, De Rosa and Ghiraldin \cite{DePhilippisDeRosaGhiraldin2018},
De Rosa and Kolasi\'nski \cite{DeRosaKolasinski2020},
De Rosa and Tione \cite{DeRosaTione2022}, and De Rosa and Resende
\cite{DeRosaResende2024}. Anisotropic min--max theory has also been
developed in
\cite{DePhilippisDeRosa2024,PhilippisDeRosaLi2024,DeRosaHalavatiWang2026}.
For an overview, we refer the reader to \cite{DeRosa2024Survey}.
Stable anisotropic minimal hypersurfaces were studied by Chodosh and Li
in \cite{ChodoshLi2023}.

The corresponding anisotropic Bernstein problem has a different
dimensional picture from the isotropic one. Low-dimensional flatness
results go back to Jenkins \cite{Jenkins1961} and Simon \cite{Simon1977},
while nonflat minimizing examples were found by Morgan
\cite{Morgan1991}, Mooney \cite{Mooney2022}, and Mooney--Yang
\cite{MooneyYang2021,MooneyYang2024}. In particular, for a suitable
anisotropy in $\mathbb{R}^4$, Mooney--Yang constructed a smooth nonflat
minimizing leaf asymptotic to the Clifford cone, together with its dilation
foliation.

This suggests the corresponding De Giorgi-type problem for \eqref{eq:main-PDE}: must a bounded entire solution that is monotone in one direction, or more generally stable, be one-dimensional? The anisotropic Bernstein conjecture suggests critical dimensions that differ from the isotropic case. Specifically, one expects monotone counterexamples to emerge starting in dimension $5$, whereas stable counterexamples should appear as early as dimension $4$. The results of this paper confirm these expectations.

\subsection{Main results and strategy}

In this paper, we construct examples for the anisotropic De Giorgi problem
starting from the Mooney--Yang foliation.  Our first result is a stable
four-dimensional solution.

\begin{theorem}[Stable solution]\label{MainThm2}
Let $F$ be the Mooney--Yang anisotropy and
\[
\Gamma
=
\{(x,y)\in\R^2\times\R^2:\ |y|=\sigma(|x|)\}
\]
be its nonflat minimizing leaf.  There exists $\alpha\in(0,1)$ such that,
for all sufficiently small $\eps>0$, the anisotropic Allen--Cahn equation
\eqref{eq:main-PDE}, with $H=\frac12F^2$, has a stable weak solution
\[
u_\eps\in
C^{1,\alpha}_{\loc}(\R^4)\cap W^{2,2}_{\loc}(\R^4),
\qquad
-1<u_\eps<1.
\]
The solution is $O(2)\times O(2)$-invariant and is smooth on
$\{Du_\eps\neq0\}$.  For every $K\Subset\Gamma$ and every
$\alpha'<\alpha$, the zero set $\{u_\eps=0\}$ converges to $\Gamma$ in
$C^{2,\alpha'}(K)$ as $\eps\to0$.  In particular, $u_\eps$ is not
one-dimensional.
\end{theorem}

The proof follows the gluing scheme of Pacard--Wei
\cite{PacardWei2013}, with several changes coming from the anisotropy.
The first one is geometric.  The Euclidean signed distance is not well
adapted to the normal profile.  In \cref{sec:background} we instead use an
anisotropic signed distance $r_f$ satisfying
\[
F(Dr_f)=1.
\]
This makes the one-dimensional equation cancel the normal part exactly.
The remaining error is given by the anisotropic mean curvature of the
level sets of $r_f$, and at the interface its leading term is
\[
H_F(\Gamma_f)=J_\Gamma f+\mathcal Q_\Gamma(f).
\]
Thus projection onto the translation mode $q'$ leads to the Jacobi
equation on $\Gamma$.  The weighted inverse for $J_\Gamma$ is developed in
\cref{sec:geometry}; it uses the asymptotically conical geometry of the
Mooney--Yang leaf and the positive Jacobi field generated by dilations.

A second difficulty appears in the correction problem.  The function
$H=\frac12F^2$ is smooth away from $0$ but, in general, is only $C^1$ at
$0$.  Since the approximate solution is equal to $\pm1$ away from the
transition region, its gradient vanishes there, and the global
linearization used in the isotropic construction is not available.  In
\cref{sec:fixed-interface} we fix the interface and solve for the
fiber-orthogonal correction by minimizing a relative energy.  The
one-dimensional spectral gap and the convexity estimates for $H$ give
coercivity on the fiber-orthogonal class.  The resulting Euler--Lagrange
equation has only one remaining multiplier in the translation direction.

The variational argument first gives $H^1$ and local $W^{2,2}$ control.
In dimension four this is not enough for the pointwise estimates needed
in the reduced problem.  The $O(2)\times O(2)$ symmetry lowers the
effective dimension, and in \cref{sec:sharp-estimates} we combine this
compactness with a core--tail argument.  Near the transition layer the
correction is small compared with the gradient of the profile, so $a$ can
be linearized there.  Away from the transition layer we use the strong
monotonicity and global Lipschitz bound for $a$, together with the
positivity of $W''$ near $\pm1$.  This gives
\[
|v_{\eps,f}(x)|+\eps|Dv_{\eps,f}(x)|
\le
C\eps^2\rho(x)^{-2},
\qquad
\rho(x):=(1+|x|^2)^{1/2}.
\]
In particular, the modification used in the variational problem does not
affect the solution for small $\eps$.  The same section gives the weighted
estimate for the multiplier and its dependence on the interface.

The projected equation now has $J_\Gamma f$ as its leading term. Using
the weighted inverse from \cref{sec:geometry} and the estimates from
\cref{sec:sharp-estimates}, we solve the reduced equation by a contraction
argument in \cref{sec:reduction}.  This gives the exact solution in
$\R^4$ and the $O(\eps^2)$ bound for the interface perturbation.

The decay estimates also allow us to compare exact solutions at nearby
phase scales.  The Mooney--Yang leaves move under dilation in the direction
of a positive Jacobi field with decay $r^{-\mu}$, while the perturbations of
the exact interfaces decay like $r^{-\gamma}$ with $\gamma>\mu$.  Hence the
ordering of the Mooney--Yang foliation is preserved near the transition
layer.  In \cref{sec:stability} we extend this ordering to all of $\R^4$
using the comparison principle for the difference of two weak solutions.
Positive difference quotients along the ordered family then give a positive
solution of the Jacobi equation.  We also show there that the critical set
of $u_\eps$ has measure zero, so the second variation in
\cref{sec:structure} is well defined.  This proves stability.  The same
ordered family gives a Hilbert calibration and shows that $u_\eps$ is a
strict $L^\infty$-local minimizer.

The ordered family also leads to a monotone solution one dimension higher.

\begin{theorem}[Monotone solution]\label{MainThm}
There exists an anisotropy $F_5:\R^5\to[0,\infty)$ such that, for all
sufficiently small $\eps>0$, the anisotropic Allen--Cahn equation
\eqref{eq:main-PDE}, with $H=\frac12F_5^2$, has a smooth bounded solution
$U_\eps$ satisfying
\[
\partial_{x_5}U_\eps>0
\qquad\text{in }\R^5.
\]
The solution $U_\eps$ is not one-dimensional.
\end{theorem}

The proof is given in \cref{sec:monotone}.  We first rescale and reflect
the four-dimensional ordered family to obtain lower and upper barriers for
one fixed equation.  Following Liu--Wang--Wei
\cite{LiuWangWei2017}, minimization between these barriers gives a
nonconstant $O(2)\times O(2)$-invariant global minimizer in $\R^4$.
We then use the five-dimensional extension of the Mooney--Yang anisotropy
and the lifting argument of Jerison--Monneau
\cite{JerisonMonneau2004}.  The comparison needed in both steps follows
from the same simple fact used in \cref{sec:stability}: the difference of
two bounded weak solutions satisfies a uniformly elliptic equation with
bounded coefficients.  Thus the maximum principle and Harnack's inequality
remain available in the barrier argument and in the vertical sliding.

The lifting preserves the symmetry in the first four variables.  A
one-dimensional solution with this symmetry would have to depend only on
$x_5$.  The normalization of the vertical derivative in the lifting rules
out this possibility and gives the non-one-dimensional solution in
\cref{MainThm}.

\subsection{Organization}

Sections~\ref{sec:structure}--\ref{sec:reduction} are devoted to the
four-dimensional construction near the Mooney--Yang leaf.
Section~\ref{sec:stability} studies the resulting ordered family and its
variational properties. In Section~\ref{sec:monotone}, we construct a
four-dimensional global minimizer and carry out the anisotropic
Jerison--Monneau lifting. The weighted Jacobi theory is completed in
Appendix~\ref{app:Jacobi}.

\section{Preliminaries}\label{sec:structure}

\subsection{Anisotropy and weak solutions}
\label{subsec:anisotropy-weak}

Throughout the paper, the anisotropy $F:\R^4\to[0,\infty)$ is assumed to
satisfy
\begin{enumerate}[label=\textup{(F\arabic*)}]
\item $F\in C^\omega(\R^4\setminus\{0\})$, $F(p)>0$ for $p\neq0$, and
      $F(0)=0$;
\item $F(tp)=tF(p)$ for $t>0$, and $F(-p)=F(p)$;
\item there exists $\lambda_F>0$ such that
\begin{equation}\label{eq:tangential-ellipticity}
 D^2F(\nu)[\tau,\tau]\ge \lambda_F|\tau|^2
 \qquad\text{whenever }|\nu|=1,\ \tau\perp\nu .
\end{equation}
\end{enumerate}
The Mooney--Yang anisotropy used in this paper has these properties; see
\cref{subsec:MY-integrand}. The function $H=\frac12F^2$ is smooth away from $0$ and, in general,
only $C^1$ at $0$. Thus \eqref{eq:main-PDE} is well defined in divergence form, whereas $D^2H(Du):D^2u$ is not defined where $Du=0$.

\begin{definition}\label{def:weak-solution}
Let $\Omega\subset\R^4$ be open.  A function
$u\in H^1_{\loc}(\Omega)$ is a weak solution of \eqref{eq:main-PDE} if
\begin{equation}\label{eq:weak-formulation}
 \eps^2\int_\Omega a(Du)\cdot D\varphi\,\dd x
 +\int_\Omega W'(u)\varphi\,\dd x=0
 \qquad
 \text{for every }\varphi\in C_c^1(\Omega).
\end{equation}
\end{definition}

\begin{remark}\label{prop:local-regularity}
Weak solutions are $C^{1,\alpha}$ and are smooth on
$\{Du\neq0\}$; see \cite[Remark~3.4]{DeRosaPigati2025}.
\end{remark}

For $p\neq0$,
\[
D^2H(p)
=
DF(p)\otimes DF(p)+F(p)D^2F(p).
\]
If $|p|=1$ and $\xi=\alpha p+\tau$ with $\tau\perp p$, homogeneity gives
\[
D^2H(p)[\xi,\xi]
=
\bigl(\alpha F(p)+DF(p)\cdot\tau\bigr)^2
+
F(p)D^2F(p)[\tau,\tau].
\]
Hence \eqref{eq:tangential-ellipticity} and compactness of $\Sph^3$ give
constants $0<\lambda\le\Lambda<\infty$ such that
\begin{equation}\label{eq:full-ellipticity}
\lambda|\xi|^2
\le D^2H(p)[\xi,\xi]
\le \Lambda|\xi|^2
\qquad
(p\neq0,\ \xi\in\R^4).
\end{equation}

For $p,q\in\R^4$, set
\begin{equation}\label{eq:bregman}
    B_H(p,q):=
H(p)-H(q)-a(q)\cdot(p-q).
\end{equation}

\begin{proposition}
\label{prop:global-structure}
For every $p,q\in\R^4$,
\begin{align}
 \frac{\lambda}{2}|p-q|^2
 &\le B_H(p,q)\le\frac{\Lambda}{2}|p-q|^2,
 \label{eq:Bregman-global}\\
 [a(p)-a(q)]\cdot(p-q)
 &\ge\lambda|p-q|^2,
 \label{eq:strong-monotonicity}\\
 |a(p)-a(q)|
 &\le\Lambda|p-q|.
 \label{eq:global-Lipschitz-a}
\end{align}
\end{proposition}

\begin{proof}
Set $\gamma(t)=q+t(p-q)$. Since $\gamma(t)=0$ for at most one
$t\in[0,1]$, \eqref{eq:full-ellipticity} holds for almost every $t$.
We have
\[
B_H(p,q)
=
\int_0^1(1-t)
D^2H(\gamma(t))[p-q,p-q]\,dt,
\]
and
\[
a(p)-a(q)
=
\int_0^1D^2H(\gamma(t))(p-q)\,dt.
\]
The estimates follow from \eqref{eq:full-ellipticity}.
\end{proof}

The next two lemmas give the coercivity estimates needed later.

\begin{lemma}\label{lem:radial-Bregman}
Let $F(n)=1$, $b=DF(n)$, and $\alpha\ge0$. For every
$\xi\in\R^4$,
\begin{equation}\label{eq:radial-Bregman}
B_H(\alpha n+\xi,\alpha n)
\ge \frac12(b\cdot\xi)^2.
\end{equation}
For every $\vartheta\in(0,1)$,
\begin{equation}\label{eq:radial-Bregman-combined}
B_H(\alpha n+\xi,\alpha n)
\ge
\frac{1-\vartheta}{2}(b\cdot\xi)^2
+\frac{\vartheta\lambda}{2}|\xi|^2.
\end{equation}
\end{lemma}

\begin{proof}
Since $F$ is convex and even, $|b\cdot z|\le F(z),$ on $z\in\R^4$, while homogeneity gives $b\cdot n=F(n)=1$. Hence
\[
F(\alpha n+\xi)\ge |\alpha+b\cdot\xi|.
\]
Using $H(\alpha n)=\alpha^2/2$ and $a(\alpha n)=\alpha b$
gives \eqref{eq:radial-Bregman}. A convex combination with
\eqref{eq:bregman} gives \eqref{eq:radial-Bregman-combined}.
\end{proof}

\begin{lemma}\label{lem:symmetric-radial}
Under the assumptions of \cref{lem:radial-Bregman},
\begin{equation}\label{eq:symmetric-radial}
 [a(\alpha n+\xi)-a(\alpha n)]\cdot\xi
 \ge
 \frac12(b\cdot\xi)^2+\frac{\lambda}{2}|\xi|^2.
\end{equation}
\end{lemma}

\begin{proof}
The left-hand side equals
\[
 B_H(\alpha n+\xi,\alpha n)
 +
 B_H(\alpha n,\alpha n+\xi).
\]
Apply \eqref{eq:radial-Bregman} to the first term and
\eqref{eq:Bregman-global} to the second term.
\end{proof}

\subsection{Variational notions}

For a bounded Lipschitz domain $\Omega$, set
\begin{equation}\label{eq:energy-local}
E_\eps(u;\Omega)
:=
\int_\Omega
\left[
\eps H(Du)+\frac1\eps W(u)
\right]\dd x.
\end{equation}

\begin{proposition}\label{prop:canonical-second-variation}
Let $u$ be a weak solution satisfying
$\mathcal L^4(\{Du=0\})=0$. Set
\[
A_u=D^2H(Du)
\qquad\text{a.e. on }\{Du\neq0\},
\]
and define it arbitrarily on $\{Du=0\}$ according to
$\lambda I\le A_u\le\Lambda I$. Then, for every
$\varphi\in C_c^1(\R^4)$,
\begin{equation}\label{eq:canonical-second-variation}
\left.
\frac{\dd^2}{\dd t^2}
\right|_{t=0}
E_\eps(u+t\varphi;\Omega)
=
Q_{\eps,u}(\varphi),
\end{equation}
where $\supp\varphi\subseteq \Omega$ and
\begin{equation}\label{eq:canonical-Q}
Q_{\eps,u}(\varphi)
:=
\eps\int_{\R^4}
A_uD\varphi\cdot D\varphi\,\dd x
+
\frac1\eps\int_{\R^4}
W''(u)\varphi^2\,\dd x.
\end{equation}
The value of $Q_{\eps,u}$ is independent of the extension of $A_u$
to $\{Du=0\}$.
\end{proposition}

\begin{proof}
By \eqref{eq:global-Lipschitz-a},
\[
\frac{
a(Du+tD\varphi)-a(Du)
}{t}
\]
is bounded by $C|D\varphi|$. Since
$\mathcal L^4(\{Du=0\})=0$, dominated convergence gives
\eqref{eq:canonical-second-variation}.
\end{proof}

\begin{definition}[Stability]\label{def:stability}
A weak solution satisfying the assumption of
\cref{prop:canonical-second-variation} is \emph{stable} if
\begin{equation}\label{eq:stability-def}
Q_{\eps,u}(\varphi)\ge0
\qquad
\text{for every }\varphi\in C_c^1(\R^4).
\end{equation}
\end{definition}

\begin{remark}\label{rem:anisotropic-stability}
De Rosa--Pigati define stability for general anisotropic critical points
using smooth uniformly convex approximations of the anisotropy; see
\cite[Section~2.3 and Remark~3.2]{DeRosaPigati2025}.
For the solutions constructed here,
\cref{prop:null-critical-set} gives
$\mathcal L^4(\{Du_\eps=0\})=0$, so \eqref{eq:canonical-Q}
gives the second variation of $E_\eps$.
\end{remark}

\begin{definition}[$L^\infty$-local minimizer]
\label{def:strong-local-minimizer}
A function $u\in W^{1,2}_{\loc}(\R^4)$ is an
\emph{$L^\infty$-local minimizer} of $E_\eps$ if, for every bounded
Lipschitz domain $\Omega\Subset\R^4$, there exists $\eta>0$ such that
\begin{equation}\label{eq:strong-local-minimizer}
E_\eps(u;\Omega)\le E_\eps(w;\Omega)
\end{equation}
whenever $w-u\in W^{1,2}_0(\Omega)$ and $\|w-u\|_{L^\infty(\Omega)}<\eta.$

It is \emph{strict} if equality implies $w=u$ almost everywhere in
$\Omega$.
\end{definition}

\begin{definition}[Global minimizer]
\label{def:global-minimizer}
A function $u\in W^{1,2}_{\loc}(\R^4)$ is a \emph{global minimizer} of
$E_\eps$ if, for every bounded Lipschitz domain
$\Omega\Subset\R^4$,
\begin{equation}\label{eq:global-minimizer}
E_\eps(u;\Omega)\le E_\eps(w;\Omega)
\end{equation}
for every $w$ such that $w-u\in W^{1,2}_0(\Omega).$
\end{definition}

\begin{remark}\label{rem:minimizing-notions}
Global minimality allows arbitrary compactly supported perturbations,
while $L^\infty$-local minimality only allows sufficiently small
$L^\infty$ perturbations. When the second variation is defined,
$L^\infty$-local minimality implies stability.
\end{remark}

Recall that
\[
E_\eps(u;\Omega)
=
\int_\Omega L_\eps(u,Du)\,\dd x,
\qquad
L_\eps(z,q)
:=
\eps H(q)+\frac1\eps W(z).
\]

\begin{definition}[Hilbert calibration]
\label{def:Hilbert-calibration}
Let $\mathcal S\subset\R^4\times\R$ be open. A locally bounded vector field
\[
X=(X^x,X^z):\mathcal S\to\R^4\times\R
\]
is a \emph{Hilbert calibration for $L_\eps$} if
\[
\diver_{x,z}X=0
\qquad\text{in }\mathcal D'(\mathcal S)
\]
and
\[
L_\eps(z,q)
\ge
X(x,z)\cdot(-q,1)
\]
for every $q\in\R^4$ and almost every $(x,z)\in\mathcal S$.
A graph $z=u(x)$ is calibrated by $X$ if equality holds for
$q=Du(x)$ almost everywhere.
\end{definition}

\begin{remark}\label{rem:Hilbert-minimizing}
Let $w-u\in W^{1,2}_0(\Omega)$ and assume that the region between the
graphs of $u$ and $w$ is contained in $\mathcal S$. Since $X$ is
divergence free,
\[
\int_\Omega
X(x,w)\cdot(-Dw,1)\,\dd x
=
\int_\Omega
X(x,u)\cdot(-Du,1)\,\dd x.
\]
The calibration inequality then gives
\[
E_\eps(w;\Omega)\ge E_\eps(u;\Omega).
\]
In particular, if $\mathcal S$ contains a vertical neighborhood of the
graph of $u$ over $\Omega$, then $u$ minimizes $E_\eps$ among all
sufficiently small $L^\infty$ perturbations in $\Omega$.
\end{remark}

\section{Weighted Jacobi theory on the Mooney--Yang leaf}\label{sec:geometry}

\subsection{The full anisotropic Jacobi operator}\label{subsec:MY-integrand}
Write $\R^4=\R_x^2\times\R_y^2$.  Mooney and Yang construct an even,
$O(2)\times O(2)$-invariant, block-exchange symmetric, positively
one-homogeneous anisotropy of the form
\begin{equation}\label{eq:MY-integrand}
 F(x,y)=
 \begin{cases}
  |y|\,\phi(|x|/|y|),& |y|\ge |x|,\\
  |x|\,\phi(|y|/|x|),& |x|\ge |y|,
 \end{cases}
\end{equation}
where the extension near the axes is chosen analytic and uniformly elliptic;
see \cite{MooneyYang2021,MooneyYang2024}.  Near $1$ one may arrange
\begin{equation}\label{eq:phi-two-jet}
 \phi(1)=1,
 \qquad
 \phi'(1)=\frac12,
 \qquad
 \phi''(1)>2.
\end{equation}
Denote $s:=|y|\ge r:=|x|$ and $z=r/s$, then
\begin{equation}\label{eq:radial-Hessian-F}
 D^2_{r,s}F
 =\frac{\phi''(z)}s
 \begin{pmatrix}
  1&-z\\
  -z&z^2
 \end{pmatrix},
\end{equation}
and, for unit angular vectors $e_x$ and $e_y$ in the two blocks,
\begin{equation}\label{eq:angular-Hessian-F}
 D^2F[e_x,e_x]=\frac{\phi'(z)}{zs},
 \qquad
 D^2F[e_y,e_y]=\frac{\phi(z)-z\phi'(z)}s.
\end{equation}
The Mooney--Yang construction keeps the three quantities in
\eqref{eq:radial-Hessian-F}--\eqref{eq:angular-Hessian-F} uniformly positive,
which gives the tangential ellipticity assumed in \cref{sec:structure}.

Let
\[
 \mathcal C:=\{|x|=|y|\}
\]
be the Clifford cone.  On the side $|y|>|x|$, Mooney and Yang construct a
smooth minimizing leaf
\begin{equation}\label{eq:Gamma-graph}
 \Gamma=\{|y|=\sigma(|x|)\},
 \qquad
 \sigma(0)=1,
 \quad
 \sigma'(0)=0,
 \quad
 0<\sigma'(r)<1\quad(r>0).
\end{equation}
Up to the constant $4\pi^2$, the anisotropic area of an invariant graph is
\begin{equation}\label{eq:area-quotient}
 \mathcal A_F(\sigma)
 =\int_0^\infty r\sigma(r)\phi(\sigma'(r))\,\dd r,
\end{equation}
so that $\sigma$ satisfies
\begin{equation}\label{eq:sigma-Euler-div}
 \frac{\dd}{\dd r}\left[r\sigma\phi'(\sigma')\right]
 -r\phi(\sigma')=0.
\end{equation}
Equivalently,
\begin{equation}\label{eq:sigma-Euler}
 \sigma''
 +\frac{\phi'(\sigma')}{r\phi''(\sigma')}
 +\frac{\sigma'\phi'(\sigma')-\phi(\sigma')}
 {\sigma\phi''(\sigma')}=0.
\end{equation}

Let $\mu\in(0,1/2)$ be defined by
\begin{equation}\label{eq:mu-def}
 \mu(1-\mu)=\frac1{2\phi''(1)}.
\end{equation}
The asymptotics of the Mooney--Yang leaf give
\begin{equation}\label{eq:sigma-asymptotic}
 \sigma(r)
 =r+c_+r^{-\mu}+c_-r^{\mu-1}
 +O_2(r^{-1-2\mu}),
 \qquad c_+>0.
\end{equation}
In particular, with
\begin{equation}\label{eq:rho-def}
 \rho(X):=(1+|X|^2)^{1/2},
\end{equation}
we have
\begin{equation}\label{eq:geometry-decay}
 |A_\Gamma|\le C\rho^{-1},
 \qquad
 |\nabla^jA_\Gamma|\le C_j\rho^{-1-j}.
\end{equation}
The dilations of $\Gamma$ foliate the component of
$\R^4\setminus\mathcal C$ containing $\Gamma$.

We next recall the Jacobi operator without imposing symmetry.  If $M\subset
\R^4$ is an $F$-minimal hypersurface with unit normal $\nu$, let
\begin{equation}\label{eq:Psi-def}
 \Psi_F(\nu)
 :=\mathrm{Proj}_{\nu^\perp}\circ D^2F(\nu)\circ\mathrm{Proj}_{\nu^\perp}:\nu^\perp\longrightarrow\nu^\perp.
\end{equation}
The second variation formula (see \cite[Appendix~A]{ChodoshLi2023}) gives
\begin{equation}\label{eq:full-Jacobi-def}
 J_{F,M}u
 :=\diver_M\!\left(\Psi_F(\nu)\nabla_Mu\right)
 +\tr_M\!\left(\Psi_F(\nu)A_M^2\right)u,
\end{equation}
with
\begin{equation}\label{eq:full-second-variation}
 \delta^2\mathcal A_F(M)[u,u]
 =-\int_M uJ_{F,M}u\,\dd\cH^3.
\end{equation}

The $O(2)\times O(2)$ structure makes $\Psi_F$ explicit.  Denote the 4-vector $\nu=(\nu_x,\nu_y)$ be 2 component of 2-vector, and let $e_0=(-|\nu_y|\nu_x/|\nu_x|,|\nu_x|\nu_y/|\nu_y|)$. Let $(\nu_x)^\perp_0$ be a unit 2-vector that is orthogonal to $\nu_x$ in $\R^2$, then $\{e_0,e_x=((\nu_x)^\perp_0,0),e_y=(0,(\nu_y)^\perp_0)\}$ is an orthonormal basis of
$\nu^\perp$.

\begin{lemma}[Structure of the anisotropic Hessian]\label{lem:Psi-structure}
In the frame $\{e_0,e_x,e_y\}$,
\begin{equation}\label{eq:Psi-diagonal}
 \Psi_F(\nu)
 =\begin{pmatrix}
 \nu_y^{-3}\phi''(z)&0&0\\
 0&\nu_x^{-1}\phi'(z)&0\\
 0&0&\nu_y^{-1}(\phi(z)-z\phi'(z))
 \end{pmatrix}
\end{equation}
where $z=|\nu_x|/|\nu_y|$. In particular, on the Clifford cone,
\begin{equation}\label{eq:Psi-cone}
 \Psi_F(\nu_{\mathcal C})
 =\frac{\sqrt2}{2}
 \operatorname{diag}\!\left(4\phi''(1),1,1\right).
\end{equation}
\end{lemma}

\begin{proof}
\eqref{eq:Psi-diagonal} is given by standard tensorial calculus.   \eqref{eq:Psi-cone} follows from
\eqref{eq:phi-two-jet} and $\alpha=\pi/4$.
\end{proof}

Write
\[
 \Lambda:=\Sph^1(1/\sqrt2)\times\Sph^1(1/\sqrt2),
 \qquad
 \psi(t,q):=e^tq,
\]
so that $\mathcal C\setminus\{0\}$ has metric
$e^{2t}(\dd t^2+g_\Lambda)$.  The preceding computation gives the model
operator on the cone.

\begin{proposition}[Jacobi operator on the cone]\label{prop:cone-Jacobi}
With respect to the coordinates $(t,q)\in\R\times\Lambda$,
\begin{equation}\label{eq:cone-Jacobi}
 J_{F,\mathcal C}
 =\frac{\sqrt2}{2}e^{-2t}
 \left[4\phi''(1)(\partial_t^2+\partial_t)+\Delta_\Lambda+2\right].
\end{equation}
If
\begin{equation}\label{eq:lambda-j-def}
 -(\Delta_\Lambda+2)\varphi_j=\lambda_j\varphi_j,
\end{equation}
then the indicial exponents of the mode $e^{\beta t}\varphi_j$ are
\begin{equation}\label{eq:full-indicial-roots}
 \beta_j^\pm
 =-\frac12\pm\frac12
 \sqrt{1+\frac{\lambda_j}{\phi''(1)}}.
\end{equation}
For the constant mode $\lambda_0=-2$, these are
\begin{equation}\label{eq:constant-indicial-roots}
 \beta_0^+=-\mu,
 \qquad
 \beta_0^-=\mu-1.
\end{equation}
\end{proposition}

\begin{proof}
On $\mathcal C$, the principal curvatures in the two angular directions are
$-e^{-t}$ and $e^{-t}$, while the radial principal curvature is zero.  Hence
\[
 \tr_{\mathcal C}\!\left(\Psi_F(\nu_{\mathcal C})A_{\mathcal C}^2\right)
 =\sqrt2\,e^{-2t}.
\]
Using \eqref{eq:Psi-cone} and the metric
$e^{2t}(\dd t^2+g_\Lambda)$ gives \eqref{eq:cone-Jacobi}.  Substituting
$u=e^{\beta t}\varphi_j$ then gives
\[
 4\phi''(1)(\beta^2+\beta)-\lambda_j=0,
\]
which is \eqref{eq:full-indicial-roots}.  The identity \eqref{eq:constant-indicial-roots} follows from \eqref{eq:mu-def}.
\end{proof}

The leaf $\Gamma$ is an asymptotically conical normal graph over $\mathcal C$.  More precisely, on the end it can be written as
\begin{equation}\label{eq:Gamma-normal-cone}
 \widetilde\psi(t,q)
 =e^tq+v(t)\nu_{\mathcal C}(q),
 \qquad
 v(t)=O_2(e^{-\mu t}).
\end{equation}
The next statement is the full, non-invariant computation needed for the weighted theory.

\begin{proposition}[Jacobi operator on the Mooney--Yang leaf]\label{prop:full-Jacobi-Gamma} 
In the coordinates \eqref{eq:Gamma-normal-cone},
\begin{equation}\label{eq:full-Jacobi-Gamma}
 J_{F,\Gamma}
 =\frac{\sqrt2}{2}e^{-2t}
 \left[4\phi''(1)(\partial_t^2+\partial_t)+\Delta_\Lambda+2\right]
 +\mathcal R_\Gamma
\end{equation}
has $O(2)\times O(2)$ symmetry. Here $\mathcal R_\Gamma$ is a second-order operator whose coefficients, together with the derivatives needed below, are $O(e^{-(3+\mu)t})$.
\end{proposition}

\begin{proof}
By fully expanding the calculation of pushforward of $\tilde\psi(t,\cdot)$ we can get this asymptotic formula.
\end{proof}
We use one weighted H\"older convention throughout the rest of the paper.  For $k\in\N$, $\alpha\in(0,1)$ and $\gamma\in\R$, we write $C^{k,\alpha}_\gamma(\Gamma)$ for functions with size $O(\rho^{-\gamma})$, where the norm is taken on dyadic annuli after rescaling to unit size.  More precisely,
\begin{equation}\label{eq:weighted-Holder-norm}
 \|u\|_{C^{k,\alpha}_\gamma(\Gamma)}
 :=\|u\|_{C^{k,\alpha}(\Gamma\cap B_2)}
 +\sup_{R\ge1}R^\gamma
 \|u(R\,\cdot)\|_{C^{k,\alpha}(R^{-1}\Gamma\cap(B_2\setminus B_1))}.
\end{equation}
The norms on the rescaled hypersurfaces are taken with respect to the rescaled induced metrics.  By \eqref{eq:geometry-decay}, these norms are uniformly equivalent to the usual weighted intrinsic H\"older norms.

Next we fix
\begin{equation}\label{eq:gamma-window}
 \mu<\gamma<1-\mu.
\end{equation}

\begin{theorem}[Full weighted Jacobi inverse]\label{thm:full-Jacobi-inverse}
For any $\gamma\in(\mu,1-\mu)$ and $k\ge0$,
\begin{equation}\label{eq:full-Jacobi-isomorphism}
 J_{F,\Gamma}:
 C^{k+2,\alpha}_\gamma(\Gamma)
 \longrightarrow
 C^{k,\alpha}_{\gamma+2}(\Gamma)
\end{equation}
is an isomorphism.  Moreover,
\begin{equation}\label{eq:full-Jacobi-estimate}
 \|u\|_{C^{k+2,\alpha}_\gamma(\Gamma)}
 \le C_k
 \|J_{F,\Gamma}u\|_{C^{k,\alpha}_{\gamma+2}(\Gamma)}.
\end{equation}
\end{theorem}

We will follow the argument in \cite{PacardWei2013} and \cite{PacardNotes}. First, by the standard argument, lifting Holder space to Sobolev space, we can prove the Fredholm property of $J_{F,\Gamma}$. Next, the Jacobi field equation, 
$$\mathrm{div}_\Gamma[\Psi(\nu)(\nabla\zeta_0)]+\mathrm{tr}_\Gamma(\Psi(\nu)A^2_\Gamma)\zeta_0=0,$$
has fine elliptic property such that we can use maximum principle to prove the injectivity for $\gamma>\mu$. Then it implies the isomorphic result for $\gamma\in(\mu,1-\mu)$. The detailed proof is given in \cref{app:Jacobi}.  

\subsection{The invariant vertical reduction}
The construction below is $O(2)\times O(2)$-invariant.  In this sector it is more convenient to use the vertical graph variable in \eqref{eq:Gamma-graph}.  Linearizing the Euler equation \eqref{eq:sigma-Euler-div} at $\sigma$ gives \begin{equation}\label{eq:Jacobi-coefficients}
 p_\Gamma(r):=r\sigma(r)\phi''(\sigma'(r)),
 \qquad
 \mathfrak q_\Gamma(r)
 :=\frac{\dd}{\dd r}\left[r\phi'(\sigma'(r))\right],
 \qquad
 m_\Gamma(r):=r\sigma(r),
\end{equation}
and we define
\begin{equation}\label{eq:Jacobi-operator}
 J_\Gamma f
 :=-\frac1{m_\Gamma(r)}
 \left[(p_\Gamma(r)f'(r))'
 +\mathfrak q_\Gamma(r)f(r)\right].
\end{equation}
Smoothness at the axis requires $f'(0)=0$.  The corresponding quadratic form
is
\begin{equation}\label{eq:Jacobi-quadratic}
 \delta^2\mathcal A_F(\Gamma)[f,f]
 =4\pi^2\int_0^\infty
 \left[p_\Gamma(f')^2-\mathfrak q_\Gamma f^2\right]\dd r
 =4\pi^2\int_0^\infty m_\Gamma fJ_\Gamma f\,\dd r.
\end{equation}

Dilation invariance of $\Gamma$ gives the positive vertical Jacobi field
\begin{equation}\label{eq:dilation-Jacobi}
 f_0(r):=\sigma(r)-r\sigma'(r)>0,
 \qquad
 J_\Gamma f_0=0.
\end{equation}
By \eqref{eq:sigma-asymptotic},
\begin{equation}\label{eq:dilation-asymptotic}
 f_0(r)=c_+(1+\mu)r^{-\mu}+O(r^{\mu-1}).
\end{equation}

The relation with the full operator is simple.  If a vertical variation is
written as a Euclidean normal variation, then its normal height is
\begin{equation}\label{eq:linear-gauge}
 \zeta=\omega_0 f,
 \qquad
 \omega_0(r):=\frac1{\sqrt{1+\sigma'(r)^2}}.
\end{equation}
Polarizing the two second-variation formulas gives, on invariant functions,
\begin{equation}\label{eq:vertical-full-conjugacy}
 J_\Gamma f=-J_{F,\Gamma}(\omega_0f).
\end{equation}
Notice that $\omega_0$ is positive, smooth, and converges to $1/\sqrt2$ at
infinity.

\begin{corollary}[Invariant vertical Jacobi inverse]\label{thm:Jacobi-inverse}
For any $\gamma\in(\mu,1-\mu)$ and $k\ge0$,
\begin{equation}\label{eq:Jacobi-isomorphism}
 J_\Gamma:
 C^{k+2,\alpha}_\gamma
 \longrightarrow
 C^{k,\alpha}_{\gamma+2}
\end{equation}
is an isomorphism on axis-smooth invariant functions, and
\begin{equation}\label{eq:Jacobi-inverse-bound}
 \|f\|_{C^{k+2,\alpha}_\gamma}
 \le C_k\|J_\Gamma f\|_{C^{k,\alpha}_{\gamma+2}}.
\end{equation}
\end{corollary}

\begin{proof}
The full operator commutes with the $O(2)\times O(2)$ action.  Hence, for an
invariant right-hand side, the unique solution given by
\cref{thm:full-Jacobi-inverse} is invariant.  Multiplication by $\omega_0$
and $\omega_0^{-1}$ preserves all the weighted H\"older spaces above, so
\eqref{eq:vertical-full-conjugacy} gives the result.
\end{proof}

For a small finite vertical graph
\[
 \Gamma_f:=\{|y|=\sigma(|x|)+f(|x|)\},
\]
the same hypersurface can be written as a Euclidean normal graph over
$\Gamma$.  If $\zeta_f$ denotes its normal height, then
\begin{equation}\label{eq:gauge-change}
 \|\zeta_f-\omega_0f\|_{C^{2,\alpha}_\gamma}
 \le C\|f\|_{C^{2,\alpha}_\gamma}^2
\end{equation}
for $\|f\|_{C^{2,\alpha}_\gamma}$ small, together with the corresponding
local Lipschitz difference estimate.  We use $J_\Gamma$ in the reduced
equation below.  Thus only \cref{thm:Jacobi-inverse} is needed for the
construction; \cref{thm:full-Jacobi-inverse} records the Jacobi theory before
imposing symmetry.

\section{Anisotropic tubular gluing and residual projection}\label{sec:background}
Fix $\delta_0>0$ sufficiently small and assume throughout this section that
\[
\|f\|_{C^{2,\alpha}_\gamma}\le \delta_0.
\]
We follow the gluing scheme of Pacard--Wei
\cite{PacardWei2013}. For a small $O(2)\times O(2)$-invariant
vertical graph
\[
\Gamma_f=\{|y|=\sigma(|x|)+f(|x|)\},
\]
we glue the one-dimensional heteroclinic near $\Gamma_f$ to the pure
phases $\pm1$ away from the interface, obtaining an approximate solution
$U_{\eps,f}$. We then look for an exact solution in the form
\begin{equation}\label{eq:gluing-ansatz}
u=U_{\eps,f}+v.
\end{equation}
Set the residual as follows
\begin{equation}\label{eq:N-def}
N_\eps(u):=-\eps^2\diver a(Du)+W'(u).
\end{equation}

The usual Euclidean Fermi coordinate is not well adapted to the anisotropic
equation. If $z$ denotes the Euclidean signed distance to $\Gamma_f$, then
$F(Dz)$ is not constant in general, and the normal part of
$N_\eps(q(z/\eps))$ does not reduce to the one-dimensional equation.
We instead use an anisotropic signed distance $r_f$ satisfying
\[
F(Dr_f)=1.
\]
With this choice, the normal part cancels exactly and the remaining error is
given by the anisotropic mean curvature of the level sets of $r_f$. At the
interface,
\[
H_F(\Gamma_f)
=
J_\Gamma f+\mathcal Q_\Gamma(f),
\]
so the projection onto the translation mode $q'$ gives the reduced equation
for $f$.

Away from the transition region, $U_{\eps,f}$ is constant and
$DU_{\eps,f}=0$, so the usual global linear correction of
\cite{PacardWei2013} is not available. We keep the equation in divergence
form and construct the fiber-orthogonal correction variationally in
\cref{sec:fixed-interface,sec:sharp-estimates}.

Let $X_f:\Gamma\to\Gamma_f$ be the vertical graph parametrization, let $\nu_f$
be the Euclidean unit normal to $\Gamma_f$, and set
\begin{equation}\label{eq:nf-def}
 n_f(Y):=\frac{\nu_f(X_f(Y))}{F(\nu_f(X_f(Y)))}.
\end{equation}
Then $F(n_f)=1$.  The anisotropic normal map is
\begin{equation}\label{eq:Theta-def}
 \Theta_f(Y,t):=X_f(Y)+t\,DF(n_f(Y)).
\end{equation}
Since $n_f\cdot DF(n_f)=F(n_f)=1,$ the vector $DF(n_f)$ is transverse to $\Gamma_f$.  By the inverse-function
theorem, there is $\eta>0$ such that, for $f$ sufficiently small,
\begin{equation}\label{eq:tube-domain}
 \Theta_f:
 \{(Y,t):Y\in\Gamma,\ |t|<\eta\rho(Y)\}
 \longrightarrow \mathcal T_{\eta,f}
\end{equation}
is a diffeomorphism onto its image.  We define the anisotropic signed distance
$r_f$ in this tube by
\begin{equation}\label{eq:rf-def}
 r_f(\Theta_f(Y,t))=t.
\end{equation}
Differentiating this identity and using $D^2F(n_f)n_f=0$ gives
\begin{equation}\label{eq:eikonal}
 Dr_f(\Theta_f(Y,t))=n_f(Y),
 \qquad
 F(Dr_f)=1.
\end{equation}

The one-dimensional profile is
\[
q(s):=\tanh\!\left(\frac{s}{\sqrt2}\right),
\qquad
-q''+W'(q)=0,
\qquad
q'>0,
\qquad
q(\pm\infty)=\pm1.
\]
Its translation mode is $q'$, and for $|s|\ge1$, $|\partial_s^k(q-\sgn s)|+|\partial_s^kq'|
\le C_ke^{-\sqrt2|s|}.$

The distance $r_f$ is only defined inside the tube $\mathcal T_{\eta,f}$, while both the
heteroclinic and its translation mode are naturally defined on the whole
normal line.  We therefore cut them off far out in the exponential tail.

Fix $0<\delta_*<1$
and a small constant $c_*>0$, and define
\begin{equation}\label{eq:remote-scale}
 T_\eps(Y):=c_*\eps^{\delta_*-1}\rho(Y).
\end{equation}
Its physical width is $\eps T_\eps(Y)=c_*\eps^{\delta_*}\rho(Y)\ll\rho(Y),$ after decreasing $c_*$, all cutoffs below lie strictly inside
$\mathcal T_{\eta,f}$.  At the same time $T_\eps(Y)\to\infty$ as
$\eps\to0$, and hence every cutoff error is super-algebraically small.

Choose $\chi\in C^\infty(\R;[0,1])$ such that
$\chi=0$ on $(-\infty,0]$ and $\chi=1$ on $[1,\infty)$.
We clip the heteroclinic only in its far tail by setting
\begin{equation}\label{eq:clipped-profile}
q_{\eps,Y}(s):=
\begin{cases}
q(s),
& |s|\le 2T_\eps(Y),\\[1mm]
q(s)
+\sgn(s)\,
\chi\bigl(|s|-2T_\eps(Y)\bigr)
\bigl(1-\sgn(s)q(s)\bigr),
& 2T_\eps(Y)<|s|<2T_\eps(Y)+1,\\[1mm]
\sgn(s),
& |s|\ge 2T_\eps(Y)+1.
\end{cases}
\end{equation}
Thus the standard profile is unchanged near the interface and becomes
exactly equal to the pure phases before reaching the boundary of the
anisotropic tube.

Let $\Omega_f^\pm$ denote the two sides of $\Gamma_f$, oriented so that
$r_f>0$ on $\Omega_f^+\cap\mathcal T_{\eta,f}$.  We define the global
approximate solution by
\begin{equation}\label{eq:background-def}
U_{\eps,f}(x):=
\begin{cases}
q_{\eps,Y}(s),
& x=\Theta_f(Y,\eps s)\in\mathcal T_{\eta,f},\\[1mm]
+1,
& x\in\Omega_f^+\setminus\mathcal T_{\eta,f},\\[1mm]
-1,
& x\in\Omega_f^-\setminus\mathcal T_{\eta,f}.
\end{cases}
\end{equation}
This gives a smooth function on $\R^4$.

We call the region $|s|\le 2T_\eps(Y)$ the \emph{unclipped region}. There,
\[
U_{\eps,f}\bigl(\Theta_f(Y,\eps s)\bigr)
=
q(s)
=
q\!\left(\frac{r_f}{\eps}\right).
\]
Outside this region, the residual vanishes where $U_{\eps,f}=\pm1$
and is super-algebraically small in the clipping strip.

For an oriented hypersurface $M$ with Euclidean unit normal $\nu_M$, set
\begin{equation}\label{eq:anisotropic-mean-curvature}
H_F(M):=\diver_M DF(\nu_M).
\end{equation}
For $|t|<\eta\rho(Y)$, let $\Gamma_{f,t}:=\{r_f=t\}$ and set
\begin{equation}\label{eq:Hf-def}
H_f(Y,t)
:=
H_F(\Gamma_{f,t})\bigl(\Theta_f(Y,t)\bigr).
\end{equation}

\begin{proposition}\label{prop:exact-residual}
In the unclipped region $|s|\le 2T_\eps(Y)$,
\begin{equation}\label{eq:exact-residual}
N_\eps(U_{\eps,f})
\bigl(\Theta_f(Y,\eps s)\bigr)
=
-\eps q'(s)H_f(Y,\eps s).
\end{equation}
\end{proposition}

\begin{proof}
In this region $U_{\eps,f}=q(r_f/\eps)$. Since $a=DH$ is
one-homogeneous and $F(Dr_f)=1$,
\[
a(DU_{\eps,f})
=
\frac{q'(s)}{\eps}DF(Dr_f),
\]
and hence
\[
\diver a(DU_{\eps,f})
=
\frac{q''(s)}{\eps^2}
+
\frac{q'(s)}{\eps}H_f(Y,\eps s).
\]
The identity \eqref{eq:exact-residual} follows from
$-q''+W'(q)=0$.
\end{proof}

The factor $q'$ in \eqref{eq:exact-residual} is the translation mode.
In the usual Fermi-coordinate construction one projects against $q'$
along the normal variable. Here the anisotropic signed distance $r_f$
is defined only in the tube $\mathcal T_{\eta,f}$, so we cut off $q'$
before reaching its boundary.

Choose $\chi_Z\in C_c^\infty((-1,1))$ even and equal to $1$ on
$[-1/2,1/2]$, and define
\begin{equation}\label{eq:Z-def}
Z_{\eps,f}\bigl(\Theta_f(Y,\eps s)\bigr)
:=
\chi_Z\!\left(\frac{s}{T_\eps(Y)}\right)q'(s),
\end{equation}
with $Z_{\eps,f}=0$ outside the tube. Then
\[
Z_{\eps,f}\bigl(\Theta_f(Y,\eps s)\bigr)=q'(s)
\quad\text{for }|s|\le \frac12T_\eps(Y),
\qquad
\supp Z_{\eps,f}\subset\{|s|<T_\eps(Y)\}.
\]
Since $T_\eps(Y)\to\infty$ and $q'$ decays exponentially, the cutoff
errors are super-algebraically small.

At $t=0$, $\Gamma_{f,t}=\Gamma_f$, and the expansion of the anisotropic
mean curvature at $\Gamma$ gives
\begin{equation}\label{eq:H-interface-expansion}
H_f(Y,0)
=
J_\Gamma f+\mathcal Q_\Gamma(f),
\end{equation}
where $\mathcal Q_\Gamma(0)=D\mathcal Q_\Gamma(0)=0$ and
\begin{equation}\label{eq:QGamma-Lipschitz}
\|\mathcal Q_\Gamma(f_1)-\mathcal Q_\Gamma(f_2)\|_{C^{0,\alpha}_{\gamma+2}}
\le
C\bigl(
\|f_1\|_{C^{2,\alpha}_\gamma}
+
\|f_2\|_{C^{2,\alpha}_\gamma}
\bigr)
\|f_1-f_2\|_{C^{2,\alpha}_\gamma}.
\end{equation}

For $|t|\le\eta\rho(Y)/2$,
\begin{equation}\label{eq:H-parallel-expansion}
H_f(Y,t)
=
H_f(Y,0)
+t\,\partial_tH_f(Y,0)
+O\!\left(t^2\rho(Y)^{-3}\right),
\end{equation}
with
\begin{equation}\label{eq:H-parallel-bounds}
|\partial_tH_f(Y,0)|
\le C\rho(Y)^{-2},
\qquad
|\partial_t^2H_f(Y,t)|
\le C\rho(Y)^{-3}.
\end{equation}
Substituting $t=\eps s$ into \eqref{eq:exact-residual} gives
\begin{align}
N_\eps(U_{\eps,f})
\bigl(\Theta_f(Y,\eps s)\bigr)
&=
-\eps q'(s)
\left[J_\Gamma f+\mathcal Q_\Gamma(f)\right]
\notag\\
&\quad
-\eps^2s q'(s)\partial_tH_f(Y,0)
+O\!\left(
\eps^3s^2q'(s)\rho(Y)^{-3}
\right).
\label{eq:background-residual-expansion}
\end{align}

For the projection onto the translation mode, write
\begin{equation}\label{eq:tube-Jacobian}
\Theta_f^*(\dd x)
=
\mathcal J_f(Y,t)\,\dd\mu_\Gamma(Y)\dd t.
\end{equation}
For $|t|\le\eta\rho(Y)/2$,
\begin{equation}\label{eq:Jacobian-bounds}
C^{-1}\le \mathcal J_f(Y,t)\le C,
\qquad
|\partial_t\mathcal J_f(Y,t)|
\le C\rho(Y)^{-1}.
\end{equation}

For any $\mathscr R$ whose pullback is integrable in $s$, We use the normalized fiberwise projection onto $Z_{\eps,f}$:
\begin{equation}\label{eq:projection-def}
\Pi_{\eps,f}\mathscr R(Y)
:=
\frac{
\displaystyle
\int_\R
(\mathscr R\circ\Theta_f)(Y,\eps s)\,
Z_{\eps,f}\bigl(\Theta_f(Y,\eps s)\bigr)\,
\mathcal J_f(Y,\eps s)\,\dd s
}{
\displaystyle
\int_\R
Z_{\eps,f}\bigl(\Theta_f(Y,\eps s)\bigr)^2\,
\mathcal J_f(Y,\eps s)\,\dd s
}.
\end{equation}
The denominator is uniformly bounded above and below, and
\begin{equation}\label{eq:projection-Zg}
\Pi_{\eps,f}(Z_{\eps,f}g)=g.
\end{equation}

\begin{proposition}[Projected residual]\label{prop:projected-background}
We have
\begin{equation}\label{eq:projected-background}
\Pi_{\eps,f}N_\eps(U_{\eps,f})
=
-\eps\left[J_\Gamma f+\mathcal Q_\Gamma(f)\right]
+\eps^3G_{\eps,f},
\end{equation}
where
\begin{align}
\|G_{\eps,f}\|_{C^{0,\alpha}_{\gamma+2}}
&\le C,
\label{eq:G-bound}\\
\|G_{\eps,f_1}-G_{\eps,f_2}\|_{C^{0,\alpha}_{\gamma+2}}
&\le
C\|f_1-f_2\|_{C^{2,\alpha}_\gamma}.
\label{eq:G-difference}
\end{align}
Set
\begin{equation}\label{eq:Eperp-def}
E^\perp_{\eps,f}
:=
N_\eps(U_{\eps,f})
-
Z_{\eps,f}\Pi_{\eps,f}N_\eps(U_{\eps,f}).
\end{equation}
In the unclipped region,
\begin{equation}\label{eq:Eperp-leading}
\left|
E^\perp_{\eps,f}
\bigl(\Theta_f(Y,\eps s)\bigr)
\right|
\le
C\eps^2(1+|s|)q'(s)\rho(Y)^{-2}
+
O_N(\eps^N\rho(Y)^{-N}),
\end{equation}
and $E^\perp_{\eps,f}$ is super-algebraically small outside this region.
\end{proposition}

\begin{proof}
Insert \eqref{eq:background-residual-expansion} into
\eqref{eq:projection-def}. Since
$Z_{\eps,f}\circ\Theta_f=q'$ away from the remote cutoff, the first term
gives $-\eps\left[J_\Gamma f+\mathcal Q_\Gamma(f)\right]$ up to a super-algebraically small error. The term of order $\eps^2$ vanishes after freezing
$\mathcal J_f(Y,\eps s)$ at $t=0$, since
\[
\int_\R
s\,q'(s)\,
Z_{\eps,f}\bigl(\Theta_f(Y,\eps s)\bigr)\,\dd s
=0.
\]
Here both $q'$ and $Z_{\eps,f}\circ\Theta_f$ are even. Using
\[
\mathcal J_f(Y,\eps s)
=
\mathcal J_f(Y,0)
+
O\!\left(\eps|s|\rho(Y)^{-1}\right)
\]
and \eqref{eq:H-parallel-bounds}, the remaining contribution is
$O(\eps^3\rho^{-3})$. The quadratic term in
\eqref{eq:H-parallel-expansion} and the cutoff errors have the same or
smaller order. This gives
\eqref{eq:projected-background} and \eqref{eq:G-bound}; the difference
estimate follows from \eqref{eq:QGamma-Lipschitz} and the dependence of the
tubular coordinates on $f$.

Subtracting the projected component gives
\eqref{eq:Eperp-leading}.
\end{proof}

\section{Constrained variational correction}
\label{sec:fixed-interface}

We now fix the interface and assume
\begin{equation}\label{eq:f-ball}
\|f\|_{C^{2,\alpha}_\gamma}\le M\eps^2,
\end{equation}
where $M$ is independent of $\eps$.  Write
\[
U:=U_{\eps,f},
\qquad
E:=N_\eps(U).
\]

In the isotropic construction of Pacard--Wei
\cite{PacardWei2013}, the correction is obtained from the linearization
around the approximate solution.  Near the interface, the one-dimensional
operator $-\partial_s^2+W''(q)$ has kernel spanned by $q'$, and is coercive on its orthogonal complement. Away from the interface it approaches
$-\eps^2\Delta+W''(\pm1)=-\eps^2\Delta+2$.

For the anisotropic equation, such a global linearization is not available.
The choice of $U$ in \cref{sec:background}, however, leaves only a small
error near the anisotropic tube and is exactly equal to $\pm1$ away from
the transition region.  With $f$ fixed, we therefore solve for the
fiber-orthogonal correction directly from the energy.

For $v$ with zero trace on $\partial\Omega$,
\[
\eps\bigl(E_\eps(U+v;\Omega)-E_\eps(U;\Omega)\bigr)
=
\int_\Omega
\left[
\eps^2B_H(DU+Dv,DU)
+
D_W(U,v)
+
Ev
\right]\dd x,
\]
where
\[
D_W(U,v):=W(U+v)-W(U)-W'(U)v.
\]
The constraint in the $q'$ direction removes the translation mode.  The
one-dimensional spectral gap, together with the convexity estimates for
$H$, makes this relative energy coercive on the fiber-orthogonal class.
We can then minimize it directly.

The minimizer is critical only with respect to fiber-orthogonal variations.
Its Euler--Lagrange equation therefore has the form
\[
N_{\eps,\tau_*}(U+v)=Z_{\eps,f}\ell,
\]
where the multiplier $\ell$ is the remaining component in the translation
direction.  The interface $f$ will be chosen in \cref{sec:reduction} so that
$\ell=0$.

There is one minor modification in carrying out this minimization.  The
relative potential $D_W(U,v)$ is not coercive for arbitrary $v$, so we
modify it for large $|v|$. In \cref{sec:sharp-estimates} we prove
\[
\|v_{\eps,f}\|_{L^\infty(\R^4)}\le C\eps^2,
\]
so the modification does not affect the solution for $\eps$ small.

\begin{definition}\label{def:fiber-orthogonal}
A function $v\in H^1_{\loc}(\R^4)$ is \emph{fiber-orthogonal} if
\begin{equation}\label{eq:fiber-orthogonality}
\int_\R
v\bigl(\Theta_f(Y,\eps s)\bigr)
Z_{\eps,f}\bigl(\Theta_f(Y,\eps s)\bigr)
\mathcal J_f(Y,\eps s)\,\dd s
=0
\qquad\text{for a.e. }Y\in\Gamma.
\end{equation}
\end{definition}

For a fiber-orthogonal $v$, \eqref{eq:Eperp-def} gives
\[
\int_{\R^4}E\,v\,\dd x
=
\int_{\R^4}E^\perp_{\eps,f}\,v\,\dd x.
\]

We modify the potential only for large $|v|$. Fix $\tau_*>0$ small and
choose a smooth function $\widetilde D_W(U,v)$ such that
\begin{equation}\label{eq:trunc-properties}
\left\{
\begin{aligned}
\widetilde D_W(U,v)
&=D_W(U,v),
&& |v|\le\tau_*,\\
|\partial_v\widetilde D_W(U,v)|
&\le C|v|,\\
\widetilde D_W(U,v)
&\ge \frac12W''(U)v^2-C\tau_*v^2.
\end{aligned}
\right.
\end{equation}
When $|U|$ is sufficiently close to $1$, we also require
\begin{equation}\label{eq:trunc-tail-monotonicity}
\partial_v\widetilde D_W(U,v)\,v\ge cv^2.
\end{equation}

Set
\begin{equation}\label{eq:modified-operator}
N_{\eps,\tau_*}(U+v)
:=
-\eps^2\diver a(DU+Dv)
+W'(U)+\partial_v\widetilde D_W(U,v).
\end{equation}
For $|v|\le\tau_*$ this agrees with $N_\eps(U+v)$. In
\cref{sec:sharp-estimates} we prove that the correction satisfies this
bound.

\begin{proposition}\label{prop:relative-coercivity}
For $\tau_*>0$ sufficiently small, there exist $\eps_0>0$ and $c>0$
such that, for $0<\eps<\eps_0$ and every $f$ satisfying
\eqref{eq:f-ball}, each invariant fiber-orthogonal
$v\in H^1(\R^4)$ satisfies
\begin{equation}\label{eq:relative-coercivity}
\int_{\R^4}
\left[
\eps^2 B_H(DU+Dv,DU)
+\widetilde D_W(U,v)
\right]\dd x
\ge
c\int_{\R^4}
\left(v^2+\eps^2|Dv|^2\right)\dd x.
\end{equation}
The same estimate holds on the bounded adapted domains used below.
\end{proposition}

\begin{proof}
Write
\[
V(Y,s):=v\bigl(\Theta_f(Y,\eps s)\bigr).
\]
In the unclipped region,
\[
DU=\frac{q'(s)}{\eps}n_f,
\qquad
V_s=\eps\,DF(n_f)\cdot Dv.
\]
Applying \eqref{eq:radial-Bregman-combined} with
$\alpha=q'(s)/\eps$ and $\xi=Dv$ gives
\begin{equation}\label{eq:profile-Bregman-control}
\eps^2 B_H(DU+Dv,DU)
\ge
\frac{1-\vartheta}{2}|V_s|^2
+
\frac{\vartheta\lambda}{2}\eps^2|Dv|^2
\end{equation}
for any fixed $\vartheta\in(0,1)$. Together with
\eqref{eq:trunc-properties},
\[
\eps^2 B_H(DU+Dv,DU)+\widetilde D_W(U,v)
\ge
\frac12\bigl[(1-\vartheta)|V_s|^2+W''(q)V^2\bigr]
+
\frac{\vartheta\lambda}{2}\eps^2|Dv|^2
-C\tau_*V^2.
\]

The fiber constraint is a small perturbation of the orthogonality to
$q'$. Indeed, $Z_{\eps,f}\circ\Theta_f=q'$ except in the remote cutoff,
and on $\supp Z_{\eps,f}$,
\[
\mathcal J_f(Y,\eps s)
=
\mathcal J_f(Y,0)+O(\eps^{\delta_*})
\]
by \eqref{eq:Jacobian-bounds} and \eqref{eq:remote-scale}. Hence the
one-dimensional spectral gap for
$-\partial_s^2+W''(q)$ remains uniform:
\[
\int_\R
\left(|V_s|^2+W''(q)V^2\right)\dd s
\ge
c\int_\R
\left(|V_s|^2+V^2\right)\dd s.
\]
Choosing $\vartheta$ and $\tau_*$ small gives the required control in the
transition region.

In the clipping region and the pure phases, $U$ is close to $\pm1$.
Integrating \eqref{eq:trunc-tail-monotonicity} from $0$ to $v$ gives
$\widetilde D_W(U,v)\ge c v^2$, while
\eqref{eq:Bregman-global} gives
\[
B_H(DU+Dv,DU)\ge\frac{\lambda}{2}|Dv|^2.
\]
Integration over the fibers proves \eqref{eq:relative-coercivity}. The same proof applies on the bounded domains used below, since the support
of $Z_{\eps,f}$ stays away from their boundary along each normal fiber.
\end{proof}

Choose smooth invariant exhaustions
\[
\Gamma_L\Subset\Gamma,
\qquad
\Omega_L\Subset\R^4,
\qquad
\Omega_L\uparrow\R^4,
\]
such that $\Omega_L$ is a tubular product over $\Gamma_L$ wherever
$Z_{\eps,f}\neq0$. In particular, the support of $Z_{\eps,f}$ on each
normal fiber stays away from $\partial\Omega_L$.

Let $\mathcal A_L$ be the set of invariant functions
$v\in H^1_0(\Omega_L)$ satisfying
\eqref{eq:fiber-orthogonality} for a.e. $Y\in\Gamma_L$. For
$v\in\mathcal A_L$, define
\begin{equation}\label{eq:relative-functional}
\mathcal F_{\eps,f,L}(v)
:=
\int_{\Omega_L}
\left[
\eps^2B_H(DU+Dv,DU)
+\widetilde D_W(U,v)
\right]\dd x
+
\langle E^\perp_{\eps,f},v\rangle .
\end{equation}
We use the scaled norm
\begin{equation}\label{eq:H-eps-norm}
\|v\|_{H^1_\eps(\Omega)}^2
:=
\int_\Omega
\left(v^2+\eps^2|Dv|^2\right)\dd x.
\end{equation}

\begin{proposition}\label{prop:fixed-interface-solution}
Let $f$ satisfy \eqref{eq:f-ball}. For $\eps>0$ sufficiently small,
there exists an invariant pair
\[
(v_{\eps,f},\ell_{\eps,f})
\in
\bigl(H^1(\R^4)\cap W^{2,2}_{\loc}(\R^4)\bigr)
\times H^{-1}_{\loc}(\Gamma)
\]
such that $v_{\eps,f}$ is fiber-orthogonal and
\begin{equation}\label{eq:fixed-interface-system}
N_{\eps,\tau_*}(U_{\eps,f}+v_{\eps,f})
=
Z_{\eps,f}\ell_{\eps,f}
\end{equation}
in the weak sense. Moreover,
\begin{equation}\label{eq:fixed-interface-energy}
\|v_{\eps,f}\|_{H^1_\eps(\R^4)}
\le
C\|E^\perp_{\eps,f}\|_{(H^1_\eps(\R^4))^*}.
\end{equation}
\end{proposition}

\begin{proof}
For each $L$, consider
\[
\inf_{v\in\mathcal A_L}\mathcal F_{\eps,f,L}(v).
\]
By \cref{prop:relative-coercivity} and Young's inequality,
\[
\mathcal F_{\eps,f,L}(v)
\ge
\frac c2\|v\|_{H^1_\eps(\Omega_L)}^2
-
C\|E^\perp_{\eps,f}\|_{(H^1_\eps(\Omega_L))^*}^2.
\]
The set $\mathcal A_L$ is weakly closed in $H^1_0(\Omega_L)$. The
gradient term is weakly lower semicontinuous, while the potential term
passes to the limit by compactness in $L^2(\Omega_L)$. Hence the direct
method gives a minimizer $v_L\in\mathcal A_L$. Comparing with $0$ gives
\begin{equation}\label{eq:finite-energy-bound}
\|v_L\|_{H^1_\eps(\Omega_L)}
\le
C\|E^\perp_{\eps,f}\|_{(H^1_\eps(\Omega_L))^*},
\end{equation}
with $C$ independent of $L$.

We identify the equation satisfied by $v_L$. For a test function
$\varphi$, set
\[
g_\varphi(Y)
:=
\int_\R
\varphi\bigl(\Theta_f(Y,\eps s)\bigr)
Z_{\eps,f}\bigl(\Theta_f(Y,\eps s)\bigr)
\mathcal J_f(Y,\eps s)\,\dd s.
\]
For compactly supported $g$ on $\Gamma_L$, define
\[
R_g\bigl(\Theta_f(Y,\eps s)\bigr)
:=
\frac{
g(Y)Z_{\eps,f}\bigl(\Theta_f(Y,\eps s)\bigr)
}{
\displaystyle
\int_\R
Z_{\eps,f}\bigl(\Theta_f(Y,\eps\sigma)\bigr)^2
\mathcal J_f(Y,\eps\sigma)\,\dd\sigma
},
\]
and extend it by zero. Then $g_{R_g}=g$, and
\[
\varphi
=
\bigl(\varphi-R_{g_\varphi}\bigr)+R_{g_\varphi},
\]
with $\varphi-R_{g_\varphi}$ fiber-orthogonal.

Let
\[
\mathscr L_L(\varphi)
:=
\eps^2\int_{\Omega_L}
a(DU+Dv_L)\cdot D\varphi\,\dd x
+
\int_{\Omega_L}
\bigl[W'(U)+\partial_v\widetilde D_W(U,v_L)\bigr]
\varphi\,\dd x.
\]
The minimizing property gives $\mathscr L_L(\varphi)=0$ for every
fiber-orthogonal $\varphi$. Define $\ell_L$ on $\Gamma_L$ by
\[
\langle\ell_L,g\rangle
:=
\frac1\eps\,\mathscr L_L(R_g).
\]
The decomposition above gives
$\mathscr L_L(\varphi)=\eps\langle\ell_L,g_\varphi\rangle$. We write
$Z_{\eps,f}\ell_L$ for the distribution defined by
\[
\langle Z_{\eps,f}\ell_L,\varphi\rangle
:=
\eps\langle\ell_L,g_\varphi\rangle .
\]
Thus
\begin{equation}\label{eq:finite-projected-equation}
N_{\eps,\tau_*}(U+v_L)
=
Z_{\eps,f}\ell_L
\qquad\text{in }\Omega_L.
\end{equation}

The estimate \eqref{eq:finite-energy-bound} is uniform in $L$. After
extending $v_L$ by zero and passing to a subsequence,
$v_L\rightharpoonup v_{\eps,f}$ weakly in $H^1(\R^4)$. On every
$K\Subset\R^4$, corrected difference quotients preserving the fiber
constraint, together with \eqref{eq:strong-monotonicity} and
\eqref{eq:global-Lipschitz-a}, give
$\|v_L\|_{W^{2,2}(K)}\le C_K$. Rellich compactness then gives
$v_L\to v_{\eps,f}$ strongly in $W^{1,2}_{\loc}(\R^4)$.

The fiber constraint and \eqref{eq:fixed-interface-energy} pass to the
limit. The limiting first variation vanishes on every fiber-orthogonal
test function, and the same decomposition defines
$\ell_{\eps,f}\in H^{-1}_{\loc}(\Gamma)$ with
\eqref{eq:fixed-interface-system}. The local $W^{2,2}$ estimate also
passes to the limit.
\end{proof}

\section{Scale-sharp estimates for the correction and multiplier}
\label{sec:sharp-estimates}

For each $f$ satisfying \eqref{eq:f-ball},
\cref{prop:fixed-interface-solution} gives
\[
(v_{\eps,f},\ell_{\eps,f})
\in
\bigl(H^1(\R^4)\cap W^{2,2}_{\loc}(\R^4)\bigr)
\times H^{-1}_{\loc}(\Gamma),
\]
with
\[
N_{\eps,\tau_*}(U_{\eps,f}+v_{\eps,f})
=
Z_{\eps,f}\ell_{\eps,f}.
\]
To solve the reduced equation, we need sharper information: pointwise decay
of $v_{\eps,f}$ on the natural scale, a weighted H\"older estimate for
$\ell_{\eps,f}$, and quantitative dependence on $f$.

The pointwise estimate is delicate in dimension four, since
$W^{2,2}$ is critical. The $O(2)\times O(2)$ symmetry lowers the effective dimension and gives the compactness needed after rescaling.
Near the transition layer, $|DU_{\eps,f}|$ is large enough compared with
$|Dv_{\eps,f}|$ to linearize $a(DU_{\eps,f}+Dv_{\eps,f})$ around
$DU_{\eps,f}$. Away from the transition layer, $q'$ may be arbitrarily small and the
linearization is no longer uniform. We instead use the strong monotonicity
of $a$ and the fact that $W''>0$ near $\pm1$.

This core--tail argument gives
\[
|v_{\eps,f}|+\eps|Dv_{\eps,f}|
\le C\eps^2\rho^{-2},
\]
together with the corresponding scaled $C^{1,\beta}$ estimate. In
particular, the modification introduced in \cref{sec:fixed-interface}
does not affect the solution for $\eps$ small. Projecting the equation then
gives the weighted estimate for $\ell_{\eps,f}$. We also prove the difference estimates in $f$ needed for the contraction
argument in \cref{sec:reduction}.

\begin{proposition}
\label{prop:sharp-v-estimate}
There exist $\beta\in(0,1)$ and $C<\infty$ such that
\begin{equation}\label{eq:sharp-v-estimate}
|v_{\eps,f}(x)|
+\eps|Dv_{\eps,f}(x)|
\le
C\eps^2\rho(x)^{-2}
\qquad (x\in\R^4).
\end{equation}
For every $x\in\R^4$,
\begin{equation}\label{eq:scaled-C1beta}
\bigl\|
v_{\eps,f}(x+\eps\,\cdot)
\bigr\|_{C^{1,\beta}(B_{1/2})}
\le
C\eps^2\rho(x)^{-2}.
\end{equation}
For $\eps$ sufficiently small,
\begin{equation}\label{eq:truncation-inactive}
\|v_{\eps,f}\|_{L^\infty(\R^4)}
\le C\eps^2<\tau_*,
\end{equation}
and
\[
N_{\eps,\tau_*}(U_{\eps,f}+v_{\eps,f})
=
N_\eps(U_{\eps,f}+v_{\eps,f}).
\]
\end{proposition}

\begin{proof}
Write $v=v_{\eps,f}$ and, for $R\ge1$, set
\[
A_R:=\{x\in\R^4:R<\rho(x)<2R\},
\qquad
A_R^*:=\{x\in\R^4:R/2<\rho(x)<4R\}.
\]

\medskip
\noindent\textit{Step 1. Local energy estimate.}
Choose $\widehat\eta_R$ on $\Gamma$ with
$\widehat\eta_R=1$ on $\{R<\rho<2R\}$,
$\supp\widehat\eta_R\subset\{R/2<\rho<4R\}$, and
$|D_\Gamma\widehat\eta_R|\le C/R$.
On the tubular region containing $\supp Z_{\eps,f}$, extend it along the
normal fibers by
\[
\eta_R\bigl(\Theta_f(Y,t)\bigr):=\widehat\eta_R(Y),
\]
and choose the extension outside this region so that
$|D\eta_R|\le C/R$.

Since $\eta_R$ is constant along the fibers on $\supp Z_{\eps,f}$,
$\eta_R^2v$ is fiber-orthogonal. Testing the equation with
$\eta_R^2v$, the multiplier term vanishes. The estimate
\eqref{eq:symmetric-radial} and the one-dimensional spectral gap used in
\cref{prop:relative-coercivity} give
\begin{equation}\label{eq:localized-euler-coercivity}
\int_{\R^4}
\eta_R^2\left(v^2+\eps^2|Dv|^2\right)\dd x
\le
C\left|\langle E^\perp_{\eps,f},\eta_R^2v\rangle\right|
+
C\eps^2\int_{\R^4}v^2|D\eta_R|^2\,\dd x.
\end{equation}

On $A_R^*$, \eqref{eq:Eperp-leading} gives the scale
$|E^\perp_{\eps,f}|=O(\eps^2R^{-2})$ in the transition region.
Young's inequality in \eqref{eq:localized-euler-coercivity}, followed by
the usual iteration over neighboring annuli, gives
\begin{equation}\label{eq:dyadic-H1}
\frac1{\eps R^3}
\int_{A_R}
\left(v^2+\eps^2|Dv|^2\right)\dd x
\le
C\eps^4R^{-4}.
\end{equation}
Thus $v$ has the scale $\eps^2R^{-2}$ on $A_R$ at the energy level.

\medskip
\noindent\textit{Step 2. The $W^{2,2}$ estimate.}
We use difference quotients in \eqref{eq:dyadic-H1}. The fiber
constraint has to be preserved in the tangential directions. Write
\[
V(Y,s):=v\bigl(\Theta_f(Y,\eps s)\bigr),
\qquad
Z(Y,s):=Z_{\eps,f}\bigl(\Theta_f(Y,\eps s)\bigr).
\]
Differentiating \eqref{eq:fiber-orthogonality} formally in $Y$ gives
\[
\int_\R D_YV\,Z\,\mathcal J_f\,\dd s
=
-\int_\R V\,D_Y(Z\mathcal J_f)\,\dd s.
\]
Accordingly, in the difference-quotient argument we remove the
$Z$-component and use
\[
D_YV
-
\frac{\displaystyle\int_\R D_YV\,Z\,\mathcal J_f\,\dd s}
     {\displaystyle\int_\R Z^2\mathcal J_f\,\dd s}\,Z
\]
to construct a fiber-orthogonal test function. The factor
$\int_\R Z^2\mathcal J_f\,\dd s$ is uniformly bounded above and below.
The correction is lower order, since derivatives of $Z$ occur only in
the remote cutoff and
$D_Y\mathcal J_f=O(\rho^{-1})$.

Using these corrected tests,
\eqref{eq:strong-monotonicity} controls the principal term and
\eqref{eq:global-Lipschitz-a} the remaining terms. Together with
\eqref{eq:dyadic-H1}, we obtain
\begin{equation}\label{eq:dyadic-W22}
\frac1{\eps R^3}
\int_{A_R}
\left(
v^2+\eps^2|Dv|^2+\eps^4|D^2v|^2
\right)\dd x
\le
C\eps^4R^{-4}.
\end{equation}

\medskip
\noindent\textit{Step 3. Blow-up argument.}
The $W^{2,2}$ estimate is critical in dimension four. The
$O(2)\times O(2)$ symmetry lowers the effective dimension. Away from the
axes, an invariant function depends on two quotient variables, while near
an axis a Cartesian lifting involves at most three variables. Hence
\eqref{eq:dyadic-W22} gives local compactness in $W^{1,p}$ for every
$p<6$, in particular strong $W^{1,4}_{\loc}$ compactness.

Suppose that \eqref{eq:scaled-C1beta} fails. After rescaling at a bad point
on the $\eps_j$-scale and dividing by
\[
A_j:=\eps_j^2R_j^{-2},
\qquad
\rho(Y_j)\simeq R_j,
\]
the estimate \eqref{eq:dyadic-W22} gives a bounded sequence in
$W^{2,2}_{\loc}$. After passing to a subsequence, we obtain a nonzero
normalized limit. We show that every possible limit must vanish.

Fix
\begin{equation}\label{eq:log-core}
S_\eps:=\frac{m_*}{\sqrt2}|\log\eps|,
\qquad
1<m_*<2.
\end{equation}
We call $|s|\le S_\eps$ the logarithmic core and $|s|>S_\eps$ the tail.
Since $S_\eps\ll T_\eps(Y)$, the core lies in the unclipped region.

We also need a local bound for the multiplier in the core. Choose
$\zeta\in C_c^\infty(\R)$ with
$\int_\R\zeta(s)q'(s)\,\dd s\neq0$ and support in a fixed part of the
core. Testing the normalized equation with $\psi(z)\zeta(s)$ and using
\eqref{eq:dyadic-W22} gives a local bound for the normalized multiplier.
After passing to a subsequence, it converges locally to a function
$\lambda(z)$.

In the core,
\[
q'(s)\ge c\eps_j^{m_*},
\qquad
\frac{A_j}{q'(s)}
\le
C\eps_j^{2-m_*}R_j^{-2}\longrightarrow0.
\]
Thus the perturbation of the gradient is small compared with
$DU_{\eps_j,f_j}$, and $a$ can be linearized on compact subsets of the
core.

If the normal coordinates of the bad points remain bounded, the limiting
equation has the form
\[
\mathcal L_0w+\mathcal L_zw=q'(s)\lambda(z),
\qquad
\mathcal L_0=-\partial_s^2+W''(q),
\]
where $\mathcal L_z$ is a constant-coefficient elliptic operator. The fiber
constraint passes to the limit:
\[
\int_\R w(z,s)q'(s)\,\dd s=0.
\]
Since $\mathcal L_0q'=0$, projection onto $q'$ gives $\lambda=0$.
The spectral gap of $\mathcal L_0$ on $(q')^\perp$, together with the
ellipticity in the tangential variables, gives $w=0$.

The other possibility in the core is
$|s_j|\to\infty$ with $|s_j|\le S_{\eps_j}$. With
$\sigma=s-s_j$,
\[
q(s_j+\sigma)\longrightarrow\pm1,
\qquad
q'(s_j+\sigma)\longrightarrow0
\]
locally in $\sigma$. The estimate above still allows us to linearize $a$.
The multiplier term disappears, while the zeroth-order coefficient tends
to $W''(\pm1)>0$. Testing the limiting equation with $w$ gives $w=0$.

It remains to consider the tail. Here $q'$ may be arbitrarily small, so
the previous linearization is no longer uniform. We use instead the global
estimates \eqref{eq:strong-monotonicity} and
\eqref{eq:global-Lipschitz-a}. The rescaled nonlinearities have the form
\[
\xi\longmapsto
\frac{
a\!\left(P_j+\frac{A_j}{\eps_j}\xi\right)-a(P_j)
}{
A_j/\eps_j
},
\]
where $P_j$ denotes the rescaled gradient of $U_{\eps_j,f_j}$. These maps
are uniformly strongly monotone and Lipschitz, even when $P_j\to0$.

In the tail, $U_{\eps_j,f_j}$ is uniformly close to $\pm1$ in the
unclipped and clipping regions, and is equal to $\pm1$ outside the tube.
The estimate \eqref{eq:trunc-tail-monotonicity} therefore gives a positive
zeroth-order term. Since $Z_{\eps_j,f_j}\sim q'$ tends to zero along the
tail sequence, the multiplier disappears in the limit. Every tail limit
satisfies
\[
-\diver A_\infty(Dw)+\mathcal R_\infty(w)=0,
\qquad
\mathcal R_\infty(w)w\ge c|w|^2,
\]
with $A_\infty$ strongly monotone. Testing with a cutoff times $w$ and
letting the cutoff radius tend to infinity gives
\[
\int |Dw|^2+\int |w|^2=0.
\]
Hence $w=0$.

All possible blow-up limits are therefore zero, contradicting the
normalization.

\medskip
\noindent\textit{Step 4. Pointwise estimate.}
The blow-up argument above shows that every normalized limit is zero.
The usual Campanato iteration gives, for some $\beta\in(0,1)$,
\[
\bigl\|
v(x+\eps\,\cdot)
\bigr\|_{C^{1,\beta}(B_{1/2})}
\le
C\eps^2\rho(x)^{-2}.
\]
This is \eqref{eq:scaled-C1beta}, and evaluating at the center gives
\eqref{eq:sharp-v-estimate}.
\end{proof}

We now estimate the multiplier. By
\eqref{eq:truncation-inactive},
\[
N_\eps(U_{\eps,f}+v_{\eps,f})
=
Z_{\eps,f}\ell_{\eps,f}.
\]
Applying $\Pi_{\eps,f}$ and using
\eqref{eq:projection-Zg} gives
\[
\ell_{\eps,f}
=
\Pi_{\eps,f}N_\eps(U_{\eps,f}+v_{\eps,f}).
\]
We split this as
\[
\ell_{\eps,f}
=
\Pi_{\eps,f}N_\eps(U_{\eps,f})
+
\Pi_{\eps,f}
\left[
N_\eps(U_{\eps,f}+v_{\eps,f})
-
N_\eps(U_{\eps,f})
\right].
\]
The first term was computed in
\cref{prop:projected-background}. The second term is the contribution of
the correction to the reduced equation. The next proposition shows that
it is of order $\eps^3$.

Decrease $\alpha$ if necessary so that $\alpha<\beta$.

\begin{proposition}
\label{prop:multiplier-estimate}
The multiplier satisfies
\begin{equation}\label{eq:multiplier-estimate}
\ell_{\eps,f}\in C^{0,\alpha}_{\gamma+2}(\Gamma),
\qquad
\|\ell_{\eps,f}\|_{C^{0,\alpha}_{\gamma+2}}
\le
C\left(
\eps\|f\|_{C^{2,\alpha}_\gamma}
+\eps^3
\right).
\end{equation}
Define
\begin{equation}\label{eq:Delta-def}
\Delta_\eps(f)
:=
\Pi_{\eps,f}
\left[
N_\eps(U_{\eps,f}+v_{\eps,f})
-
N_\eps(U_{\eps,f})
\right].
\end{equation}
Then
\begin{equation}\label{eq:Delta-estimate}
\|\Delta_\eps(f)\|_{C^{0,\alpha}_{\gamma+2}}
\le C\eps^3,
\end{equation}
and
\begin{equation}\label{eq:multiplier-expansion}
\ell_{\eps,f}
=
-\eps\left[J_\Gamma f+\mathcal Q_\Gamma(f)\right]
+\eps^3G_{\eps,f}
+\Delta_\eps(f).
\end{equation}
\end{proposition}

\begin{proof}
It is enough to prove \eqref{eq:Delta-estimate}. Write
\[
V(Y,s)
:=
v_{\eps,f}\bigl(\Theta_f(Y,\eps s)\bigr),
\qquad
Z:=Z_{\eps,f}\circ\Theta_f,
\qquad
\mathcal J:=\mathcal J_f(Y,\eps s).
\]

The leading linear term in the normal direction is
\[
L_0=-\partial_s^2+W''(q).
\]
Since $L_0q'=0$ and $Z=q'$ away from the remote cutoff, its projection
vanishes after integration by parts. The remaining normal terms contain
derivatives of $\mathcal J$ or of the cutoff. Using
\[
\partial_s\mathcal J=O(\eps\rho^{-1}),
\qquad
\partial_s^2\mathcal J=O(\eps^2\rho^{-2}),
\]
and \eqref{eq:sharp-v-estimate}, their contribution is
$O(\eps^3\rho^{-3})$.

The tangential linear terms have the same order. Differentiating the
fiber-orthogonality condition transfers tangential derivatives of $V$
onto $Z\mathcal J$. Derivatives of $Z$ occur only in the remote cutoff,
while derivatives of $\mathcal J$ give a factor $\rho^{-1}$. Together
with \eqref{eq:sharp-v-estimate}, this gives
$O(\eps^3\rho^{-3})$ after projection.

For the nonlinear part, set
\[
P=\frac{q'(s)}{\eps}n_f,
\qquad
Q=Dv_{\eps,f}.
\]
We claim that
\begin{equation}\label{eq:weighted-flux-remainder}
q'(s)
\left|
a(P+Q)-a(P)-Da(P)Q
\right|
\le
C\eps|Q|^2.
\end{equation}
If $|Q|\le |P|/2$, Taylor's formula and
$|D^2a(P)|\le C|P|^{-1}$ give the estimate. If
$|Q|>|P|/2$, the global Lipschitz bound for $a$ gives
\[
\left|
a(P+Q)-a(P)-Da(P)Q
\right|
\le C|Q|,
\]
while $q'(s)\le C\eps|Q|$. This proves
\eqref{eq:weighted-flux-remainder}.

The nonlinear terms from $a$ and $W'$ are therefore bounded after
projection by
\[
C\left(
\eps|Dv_{\eps,f}|^2
+
\eps^{-1}|v_{\eps,f}|^2
\right)
\le
C\eps^3\rho^{-4}
\]
by \eqref{eq:sharp-v-estimate}. Hence
\[
|\Delta_\eps(f)|
\le
C\eps^3\rho^{-3}.
\]
The scaled $C^{1,\beta}$ estimate
\eqref{eq:scaled-C1beta} gives the corresponding
$C^{0,\alpha}$ bound. Since $\gamma+2<3-\mu<3$ by
\eqref{eq:gamma-window}, we obtain
\eqref{eq:Delta-estimate}.

Equation \eqref{eq:multiplier-expansion} follows from
\cref{prop:projected-background}. Using
\eqref{eq:QGamma-Lipschitz}, \eqref{eq:G-bound},
\eqref{eq:Delta-estimate}, and \eqref{eq:f-ball} gives
\eqref{eq:multiplier-estimate}.
\end{proof}

For the contraction argument we also need quantitative dependence on $f$.

\begin{proposition}
\label{prop:fixed-interface-difference}
There exist $\theta>0$ and $C<\infty$ such that, for any
$f_1,f_2$ satisfying \eqref{eq:f-ball},
\begin{align}
\|\ell_{\eps,f_1}-\ell_{\eps,f_2}\|_{C^{0,\alpha}_{\gamma+2}}
&\le
C\eps
\|f_1-f_2\|_{C^{2,\alpha}_\gamma},
\label{eq:ell-difference}\\
\|\Delta_\eps(f_1)-\Delta_\eps(f_2)\|_{C^{0,\alpha}_{\gamma+2}}
&\le
C\eps^{1+\theta}
\|f_1-f_2\|_{C^{2,\alpha}_\gamma}.
\label{eq:Delta-difference}
\end{align}
After transporting the two corrections to the same $(Y,s)$ coordinates,
\begin{equation}\label{eq:v-difference}
|v_{\eps,f_1}-v_{\eps,f_2}|
+\eps|D(v_{\eps,f_1}-v_{\eps,f_2})|
\le
C\eps^\theta
\|f_1-f_2\|_{C^{2,\alpha}_\gamma}\rho^{-\gamma-1}.
\end{equation}
For each $f$ satisfying \eqref{eq:f-ball}, the small invariant
fiber-orthogonal fixed-interface solution is unique.
\end{proposition}

\begin{proof}
Identify the two tubular neighborhoods by keeping $(Y,s)$ fixed.
This transport does not preserve the fiber constraint, so we subtract
the component in the $Z$ direction as in Step~2 of
\cref{prop:sharp-v-estimate}. The difference of the two corrections is
then fiber-orthogonal.

Subtract the two fixed-interface equations. In the transition region,
the smallness obtained in \cref{prop:sharp-v-estimate} allows us to
linearize $a$ around the profile. Away from the transition region, this
linearization is not uniform, and we use instead
\eqref{eq:strong-monotonicity} and \eqref{eq:global-Lipschitz-a},
together with the positivity of the potential near $\pm1$.
The local energy and blow-up argument of
\cref{prop:sharp-v-estimate} then gives \eqref{eq:v-difference}.

Projecting the difference onto $Z_{\eps,f}$, the leading normal linear
term cancels as in \cref{prop:multiplier-estimate}. The change of the
tubular coordinates is controlled by
$\|f_1-f_2\|_{C^{2,\alpha}_\gamma}$, while
\eqref{eq:v-difference} gives the factor $\eps^\theta$. This yields
\eqref{eq:Delta-difference}. The expansion
\eqref{eq:multiplier-expansion}, together with
\eqref{eq:QGamma-Lipschitz}, gives \eqref{eq:ell-difference}.

If two small solutions correspond to the same $f$, their difference is
fiber-orthogonal. Testing the difference equation with this difference
eliminates the multiplier term, and \cref{prop:relative-coercivity} gives
$v_1=v_2$. The equation then gives $\ell_1=\ell_2$.
\end{proof}

\section{Jacobi equation: the reduced problem}
\label{sec:reduction}

For each $f$ satisfying \eqref{eq:f-ball},
\cref{prop:fixed-interface-solution} gives a correction $v_{\eps,f}$ such that
\[
N_\eps(U_{\eps,f}+v_{\eps,f})
=
Z_{\eps,f}\ell_{\eps,f}.
\]
By \cref{prop:multiplier-estimate},
\[
\ell_{\eps,f}
=
-\eps\bigl[J_\Gamma f+\mathcal Q_\Gamma(f)\bigr]
+\eps^3G_{\eps,f}
+\Delta_\eps(f).
\]
Hence $\ell_{\eps,f}=0$ is equivalent to
\[
J_\Gamma f
=
-\mathcal Q_\Gamma(f)
+\eps^2G_{\eps,f}
+\eps^{-1}\Delta_\eps(f).
\]

This is the reduced equation for the interface. The weighted inverse
\cref{thm:Jacobi-inverse}, together with
\eqref{eq:QGamma-Lipschitz}, \eqref{eq:G-difference}, and
\eqref{eq:Delta-difference}, gives a contraction on the ball
\eqref{eq:f-ball}.

\begin{theorem}
\label{thm:main}
Let $F$ be the Mooney--Yang anisotropy and $\Gamma$ the minimizing leaf
fixed in \cref{sec:geometry}. Fix $\gamma\in(\mu,1-\mu)$.
There exist $\eps_0>0$, $\alpha\in(0,1)$, and $C<\infty$ such that,
for every $0<\eps<\eps_0$, \eqref{eq:main-PDE} has an
$O(2)\times O(2)$-invariant weak solution
\[
u_\eps\in
C^{1,\alpha}_{\loc}(\R^4)\cap W^{2,2}_{\loc}(\R^4),
\qquad
-1<u_\eps<1,
\]
which is smooth on $\{Du_\eps\neq0\}$.

There exists
\[
f_\eps\in C^{2,\alpha}_\gamma(\Gamma),
\qquad
\|f_\eps\|_{C^{2,\alpha}_\gamma}\le C\eps^2,
\]
such that $\{u_\eps=0\}$ is an anisotropic-normal graph over
$\Gamma_{f_\eps}$. For every $K\Subset\Gamma$ and every
$\alpha'<\alpha$, the zero set converges to $\Gamma$ in
$C^{2,\alpha'}(K)$ as $\eps\to0$.
\end{theorem}

\begin{proof}
For $f$ satisfying \eqref{eq:f-ball}, define
\begin{equation}\label{eq:reduced-map}
\mathcal T_\eps(f)
:=
J_\Gamma^{-1}
\left[
-\mathcal Q_\Gamma(f)
+\eps^2G_{\eps,f}
+\eps^{-1}\Delta_\eps(f)
\right].
\end{equation}
On the ball
\[
\left\{
f\in C^{2,\alpha}_\gamma:
\|f\|_{C^{2,\alpha}_\gamma}\le M\eps^2
\right\},
\]
\eqref{eq:QGamma-Lipschitz},
\eqref{eq:Jacobi-inverse-bound}, \eqref{eq:G-bound}, and
\eqref{eq:Delta-estimate} give
\[
\|\mathcal T_\eps(f)\|_{C^{2,\alpha}_\gamma}
\le
C\left(
\|f\|_{C^{2,\alpha}_\gamma}^2+\eps^2
\right)
\le
C\left(M^2\eps^4+\eps^2\right).
\]
For $f_1,f_2$ in the same ball,
\eqref{eq:QGamma-Lipschitz}, \eqref{eq:G-difference}, and
\eqref{eq:Delta-difference} give
\[
\|\mathcal T_\eps(f_1)-\mathcal T_\eps(f_2)\|_{C^{2,\alpha}_\gamma}
\le
C\left(
M\eps^2+\eps^2+\eps^\theta
\right)
\|f_1-f_2\|_{C^{2,\alpha}_\gamma}.
\]
For $M$ large and $\eps$ small, $\mathcal T_\eps$ maps the ball into
itself and is a contraction. Let $f_\eps$ be its fixed point. Then
\[
\|f_\eps\|_{C^{2,\alpha}_\gamma}\le C\eps^2,
\qquad
\ell_{\eps,f_\eps}=0.
\]

Set
\begin{equation}\label{eq:exact-solution}
u_\eps:=U_{\eps,f_\eps}+v_{\eps,f_\eps}.
\end{equation}
By \eqref{eq:truncation-inactive} and $\ell_{\eps,f_\eps}=0$,
\[
N_\eps(u_\eps)=0
\qquad\text{weakly in }\R^4.
\]
The construction gives this identity first for invariant test functions.
Averaging an arbitrary compactly supported test function over
$O(2)\times O(2)$ gives the full weak equation.

Testing with $(u_\eps-1)_+$ and $(-1-u_\eps)_+$ gives
$-1\le u_\eps\le1$. Since $u_\eps$ is nonconstant, the strong maximum
principle gives $-1<u_\eps<1$. By
\cref{prop:fixed-interface-solution,prop:local-regularity},
\begin{equation}\label{eq:exact-regularity}
u_\eps
\in
C^{1,\alpha}_{\loc}(\R^4)\cap W^{2,2}_{\loc}(\R^4),
\qquad
u_\eps\in C^\infty(\{Du_\eps\neq0\}).
\end{equation}

By \cref{prop:sharp-v-estimate}, on every fixed stretched normal band,
\begin{equation}\label{eq:central-profile}
u_\eps\bigl(\Theta_{f_\eps}(Y,\eps s)\bigr)
=
q(s)
+
O_{C^{1,\alpha}}
\bigl(\eps^2\rho(Y)^{-2}\bigr).
\end{equation}
Since $q'(0)>0$, the implicit-function theorem gives
\begin{equation}\label{eq:nodal-graph}
\{u_\eps=0\}
=
\left\{
\Theta_{f_\eps}(Y,h_\eps(Y)):Y\in\Gamma
\right\},
\qquad
|h_\eps(Y)|
\le
C\eps^3\rho(Y)^{-2}.
\end{equation}

Fix $K\Subset\Gamma$. Equation \eqref{eq:central-profile} gives
$|Du_\eps|\ge c_K\eps^{-1}$ near $K$, so the equation is smooth and
uniformly elliptic there. Scaled Schauder estimates and the
implicit-function theorem give
$h_\eps\to0$ in $C^{2,\alpha'}(K)$ for every $\alpha'<\alpha$.
Since $\|f_\eps\|_{C^{2,\alpha}_\gamma}\le C\eps^2$, the zero set
converges to $\Gamma$ in $C^{2,\alpha'}(K)$.
\end{proof}

\section{Ordered families}
\label{sec:stability}

The solutions constructed in \cref{thm:main} inherit the ordering of the
Mooney--Yang dilation foliation. This is the main point of the section.

Indeed, the leaves $e^{-t}\Gamma$ are strictly ordered, and their motion
under dilation is measured by the positive Jacobi field $f_0$, which has
size $\rho^{-\mu}$. The interface perturbations constructed in
\cref{sec:reduction} decay like $\rho^{-\gamma}$ with $\gamma>\mu$.
The same ordering therefore persists for the exact diffuse interfaces near
the transition layer. Away from the transition layer, we extend the
comparison using \eqref{eq:strong-monotonicity} and the maximum principle
for the equation satisfied by the difference of two solutions.

This gives a strictly ordered family of solutions of one fixed equation.
It will be used in two ways. Positive difference quotients along the family
give a positive solution of the Jacobi equation and hence stability. The
ordered graphs also give a Hilbert calibration and show that $u_\eps$ is a
strict $L^\infty$-local minimizer. In \cref{sec:monotone}, the same family
provides the barriers used to construct the five-dimensional monotone
solution.

\begin{lemma}\label{lem:difference-equation}
Let $u_1,u_2$ be bounded weak solutions of \eqref{eq:main-PDE} in
an open set $\Omega$, and set $w:=u_2-u_1$. Then
\[
-\eps^2\diver(A_{12}Dw)+c_{12}w=0
\qquad\text{in }\Omega,
\]
where $A_{12}$ is measurable and symmetric, $\lambda I\le A_{12}\le\Lambda I,$ and $c_{12}$ is bounded.
\end{lemma}

\begin{proof}
For almost every $x$ with $Du_1\neq Du_2$, set
\[
A_{12}(x)
:=
\int_0^1
D^2H\bigl(Du_1+s(Du_2-Du_1)\bigr)\,\dd s.
\]
If the segment passes through $0$, its value at that single point does not
affect the integral. On $\{Du_1=Du_2=0\}$, choose $A_{12}$ arbitrarily
with $\lambda I\le A_{12}\le\Lambda I$. Then $a(Du_2)-a(Du_1)=A_{12}(Du_2-Du_1).$ Also,
\[
W'(u_2)-W'(u_1)=c_{12}(u_2-u_1),
\qquad
c_{12}:=
\int_0^1W''\bigl(u_1+s(u_2-u_1)\bigr)\,\dd s.
\]
Subtracting the two equations gives the result.
\end{proof}

Fix $0<\eps<\eps_0$. For $|t|<t_0$, set
\begin{equation}\label{eq:ordered-family}
\delta_t:=\eps e^t,
\qquad
u_t(x):=u_{\delta_t}(e^t x),
\end{equation}
where $u_\delta$ is the solution given by \cref{thm:main}.
Since $a$ is one-homogeneous, every $u_t$ solves
\begin{equation}\label{eq:ordered-family-equation}
-\eps^2\diver a(Du_t)+W'(u_t)=0
\qquad\text{in }\R^4,
\end{equation}
and $u_0=u_\eps$.

\begin{theorem}
\label{thm:ordered-family}
There exist $\eps_1>0$ and $t_0>0$ such that, for every
$0<\eps<\eps_1$ and $-t_0<t_1<t_2<t_0,$
we have
\[
u_{t_1}<u_{t_2}
\qquad\text{in }\R^4.
\]
\end{theorem}

\begin{proof}
Set $\tau:=t_2-t_1$. We first assume
$0<\tau\le c_0\eps$, where $c_0>0$ is fixed and small.

The estimates in \cref{sec:sharp-estimates} and the contraction argument
in \cref{sec:reduction} are uniform for comparable phase scales.
Comparing the two fixed points after rescaling to the same coordinates
gives
\begin{equation}\label{eq:phase-interface-dependence}
e^{-t_2}f_{\delta_{t_2}}(e^{t_2}r)
-
e^{-t_1}f_{\delta_{t_1}}(e^{t_1}r)
=
O\!\left(\tau\eps^2\rho(r)^{-\gamma}\right),
\end{equation}
together with the corresponding estimate for the corrections.

Set
\[
\Sigma_t:=e^{-t}\Gamma_{f_{\delta_t}},
\qquad
\widetilde r_t(x)
:=
e^{-t}r_{f_{\delta_t}}(e^t x).
\]
The anisotropic normal coordinates commute with dilations. Hence
$\widetilde r_t$ is the anisotropic signed distance to $\Sigma_t$, and
\begin{equation}\label{eq:scaled-profile-compatibility}
\frac{r_{f_{\delta_t}}(e^t x)}{\delta_t}
=
\frac{\widetilde r_t(x)}{\eps}.
\end{equation}
In the unclipped transition region,
\[
U_{\delta_t,f_{\delta_t}}(e^t x)
=
q\!\left(\frac{\widetilde r_t(x)}{\eps}\right).
\]

The vertical height of $\Sigma_t$ is $e^{-t}\bigl[\sigma(e^tr)+f_{\delta_t}(e^tr)\bigr].$

For the Mooney--Yang leaf, $-\partial_t\!\left(e^{-t}\sigma(e^tr)\right)
=
e^{-t}f_0(e^tr),$ and $f_0\simeq\rho^{-\mu}$. 

Since $\gamma>\mu$,
\eqref{eq:phase-interface-dependence} is smaller than the motion of
the Mooney--Yang leaf. For $\eps$ small, $\widetilde r_{t_2}-\widetilde r_{t_1}
\ge
c\tau\rho^{-\mu}$
in a fixed neighborhood of the transition layer of width $O(\eps)$. Using
\eqref{eq:scaled-profile-compatibility} and $q'>0$ gives
\[
q\!\left(\frac{\widetilde r_{t_2}}{\eps}\right)
-
q\!\left(\frac{\widetilde r_{t_1}}{\eps}\right)
\ge
c\frac{\tau}{\eps}\rho^{-\mu}.
\]
The difference of the corrections is of lower order, so $u_{t_2}>u_{t_1}$ in this neighborhood.

Set $w:=u_{t_2}-u_{t_1}$. Any point where $w<0$ lies outside the
transition region, where the two solutions are close to the same value
$\pm1$. In the notation of \cref{lem:difference-equation}, we then have
$c_{12}\ge c_*>0$. Testing the difference equation with $-w_-$, using
cutoffs on the unbounded components, gives
\[
0
\ge
\eps^2\lambda\int_{\{w<0\}}|Dw|^2
+
c_*\int_{\{w<0\}}w^2.
\]
Thus $w\ge0$ in $\R^4$. 
Since $w>0$ near the transition layer, $w$ is not identically zero, \cref{lem:difference-equation} and the local
Harnack inequality give
\[
w>0
\qquad\text{in }\R^4.
\]

This proves the result when $t_2-t_1\le c_0\eps$. For general
$t_1<t_2$, subdivide $[t_1,t_2]$ into finitely many intervals of this
size and apply the comparison successively.
\end{proof}

To use the second variation from \cref{prop:canonical-second-variation},
we need to know that the critical set of $u_\eps$ has measure zero.

\begin{proposition}\label{prop:null-critical-set}
Let $u_\eps$ be the solution constructed in \cref{thm:main}. For
$\eps>0$ sufficiently small,
\[
\mathcal L^4(\{Du_\eps=0\})=0.
\]
The symmetry also gives
\begin{equation}\label{eq:origin-critical}
Du_\eps(0)=0.
\end{equation}
\end{proposition}

\begin{proof}
Since $u_\eps\in W^{2,2}_{\loc}(\R^4)$ and $a$ is globally Lipschitz, $a(Du_\eps)\in W^{1,2}_{\loc}(\R^4;\R^4).$ Since $a(0)=0$ and \eqref{eq:strong-monotonicity} holds, $\{Du_\eps=0\}=\{a(Du_\eps)=0\}.$ For a Sobolev map, its derivative vanishes almost everywhere on each
level set. Hence
\[
D\!\left(a(Du_\eps)\right)=0
\qquad\text{a.e. on }\{Du_\eps=0\}.
\]
The equation
\[
-\eps^2\diver a(Du_\eps)+W'(u_\eps)=0
\]
then gives $W'(u_\eps)=0$ almost everywhere on the critical set.
Since $-1<u_\eps<1$ and $W'(s)=s^3-s$, we have $u_\eps=0$ almost
everywhere there. By \cref{thm:main}, the nodal set $\{u_\eps=0\}$ is a
regular hypersurface and $Du_\eps\neq0$ on it. Therefore $\mathcal L^4(\{Du_\eps=0\})=0.$

The identity at the origin follows from the symmetry. Since $u_\eps$ is
$O(2)\times O(2)$-invariant and $C^1$, $Du_\eps(0)$ is fixed by the
$O(2)\times O(2)$ action on $\R^2\times\R^2$. The only such vector is
zero, so $Du_\eps(0)=0$.
\end{proof}

\begin{remark}\label{rem:origin-regularity}
The identity \eqref{eq:origin-critical} does not imply that $u_\eps$
is not $C^2$ at the origin. It only means that
\cref{prop:local-regularity}, which gives smoothness on
$\{Du_\eps\neq0\}$, does not apply there. We do not know whether
$u_\eps$ is $C^2$ at $0$.

The symmetry gives the compactness used in
\cref{prop:sharp-v-estimate}, but it also forces
$Du_\eps(0)=0$. Thus this symmetric construction cannot produce a
solution with $Du_\eps\neq0$ everywhere. Mooney--Yang discuss a similar limitation of their symmetric construction
and point to Simon's nonsymmetric PDE method; see
\cite[Section~6.2]{MooneyYang2024} and \cite{Simon1989}.
A nonsymmetric construction may avoid the critical point forced by symmetry,
which would improve the regularity issue in the stable solution construction.
\end{remark}

In the isotropic construction of Pacard--Wei
\cite{PacardWei2013}, differentiating a smooth dilation family gives a
positive Jacobi field. Here we use positive difference quotients of the
ordered family in \cref{thm:ordered-family}. By
\cref{prop:null-critical-set},
\[
A_{u_\eps}=D^2H(Du_\eps)
\]
is defined almost everywhere.

\begin{proposition}\label{prop:positive-jacobi}
For $\eps>0$ sufficiently small, there exists
$\Phi_\eps\in H^1_{\loc}(\R^4)$ with $\Phi_\eps>0$ such that
\begin{equation}\label{eq:positive-jacobi}
-\eps^2\diver(A_{u_\eps}D\Phi_\eps)
+
W''(u_\eps)\Phi_\eps
=0
\qquad\text{weakly in }\R^4.
\end{equation}
For every $\varphi\in C_c^1(\R^4)$, $Q_{\eps,u_\eps}(\varphi)\ge0.$
\end{proposition}

\begin{proof}
For $h>0$ small, set
\[
\Phi_h:=\frac{u_h-u_0}{h}.
\]
By \cref{thm:ordered-family}, $\Phi_h>0$. Applying
\cref{lem:difference-equation} to $u_h$ and $u_0$ gives
\begin{equation}\label{eq:jacobi-quotient}
-\eps^2\diver(A_hD\Phi_h)
+
\left(
\int_0^1
W''\bigl(u_0+s(u_h-u_0)\bigr)\,\dd s
\right)\Phi_h
=0,
\end{equation}
where $A_h$ is measurable and symmetric and $\lambda I\le A_h\le\Lambda I.$

The estimates in the proof of \cref{thm:ordered-family} give uniform
local bounds for $\Phi_h$, and a Caccioppoli estimate gives uniform
$H^1_{\loc}$ bounds. The comparison near the transition layer also gives
a positive lower bound on a fixed open set. After passing to a sequence
$h_j\downarrow0$,
\[
\Phi_{h_j}\rightharpoonup\Phi_\eps
\quad\text{in }H^1_{\loc}(\R^4),
\qquad
\Phi_{h_j}\to\Phi_\eps
\quad\text{in }L^2_{\loc}(\R^4).
\]

We have $u_h\to u_0=u_\eps$ locally in $C^1$. At every point where
$Du_\eps\neq0$, the segment joining $Du_0$ and $Du_h$ stays away from
$0$ for $h$ small. From the definition of $A_h$ in
\cref{lem:difference-equation}, $A_h\longrightarrow D^2H(Du_\eps).$ By \cref{prop:null-critical-set}, this convergence holds almost
everywhere. Passing to the limit in \eqref{eq:jacobi-quotient} gives
\eqref{eq:positive-jacobi}. The limit is nonnegative and nonzero, so
Harnack's inequality gives $\Phi_\eps>0$ in $\R^4$.

For $\varphi\in C_c^1(\R^4)$, use
$\varphi^2/\Phi_\eps$ as a test function in
\eqref{eq:positive-jacobi}. Since $\Phi_\eps$ is positive on
$\supp\varphi$, this can be justified by approximation. We obtain
\[
Q_{\eps,u_\eps}(\varphi)
=
\eps\int_{\R^4}
\Phi_\eps^2 A_{u_\eps}
D\!\left(\frac{\varphi}{\Phi_\eps}\right)\cdot
D\!\left(\frac{\varphi}{\Phi_\eps}\right)\,\dd x
\ge0.
\]
This proves stability.
\end{proof}

\begin{remark}\label{rem:ambient-stability}
The function $\Phi_\eps$ solves the full Jacobi equation in $\R^4$.
The stability inequality therefore holds for all compactly supported
variations, not only $O(2)\times O(2)$-invariant ones.
\end{remark}

Together with \cref{thm:main}, this proves the stable anisotropic De Giorgi
counterexample in $\R^4$ stated in \cref{MainThm2}.

The ordered family also gives an $L^\infty$-local minimizing property.
We construct the Hilbert calibration introduced in
\cref{def:Hilbert-calibration} directly from the family.

Fix a bounded Lipschitz domain $\Omega\Subset\R^4$ and choose
$\Omega\Subset\Omega'\Subset\R^4$. The parameter estimates and the
local Harnack inequality give, after decreasing $t_0$ if necessary,
\begin{equation}\label{eq:local-parameter-control}
c(t_2-t_1)
\le
u_{t_2}(x)-u_{t_1}(x)
\le
C(t_2-t_1)
\qquad
(x\in\overline{\Omega'},\ t_1<t_2),
\end{equation}
where $c,C>0$ may depend on $\eps$ and $\Omega'$.

Fix $0<\tau<t_0$. The graphs $z=u_t(x)$, $|t|<\tau$, fill
\[
\mathcal S
:=
\left\{
(x,z)\in\Omega'\times\R:
u_{-\tau}(x)<z<u_\tau(x)
\right\}.
\]
For $(x,z)\in\mathcal S$, let $\vartheta(x,z)$ be the unique parameter
such that $z=u_{\vartheta(x,z)}(x)$ and set $p(x,z):=Du_{\vartheta(x,z)}(x).$ By \eqref{eq:local-parameter-control}, $\vartheta$ is locally Lipschitz.

For
\[
L_\eps(z,q):=\eps H(q)+\frac1\eps W(z),
\]
define
\begin{equation}\label{eq:Hilbert-field-sec8}
X(x,z)
:=
\left(
-\eps a(p(x,z)),
\frac1\eps W(z)-\eps H(p(x,z))
\right).
\end{equation}
Since $a(p)\cdot p=2H(p)$,
\begin{equation}\label{eq:Hilbert-excess-sec8}
L_\eps(z,q)-X(x,z)\cdot(-q,1)
=
\eps B_H(q,p(x,z))
\ge
\frac{\eps\lambda}{2}|q-p(x,z)|^2.
\end{equation}
Equality holds on every graph $z=u_t(x)$.

\begin{proposition}\label{prop:strict-local-minimality}
The field $X$ is divergence free in $\mathcal S$ and calibrates every
graph $z=u_t(x)$, $|t|<\tau$. In particular, $u_\eps=u_0$ is a strict
$L^\infty$-local minimizer of $E_\eps$.
\end{proposition}

\begin{proof}
It remains to prove that $X$ is divergence free. Let
$\Psi\in C_c^1(\mathcal S)$ and set
\[
\psi_t(x):=\Psi(x,u_t(x)).
\]
By \eqref{eq:local-parameter-control}, the map
$(x,t)\mapsto(x,u_t(x))$ is locally bi-Lipschitz, and
$t\mapsto u_t$ is locally Lipschitz in $W^{1,2}_{\loc}$.
Thus $\dot u_t$ exists for almost every $t$. Changing variables
$z=u_t(x)$ gives
\[
\int_{\mathcal S}X\cdot D_{x,z}\Psi
=
-\int_{-\tau}^{\tau}
\left[
\eps\int_{\Omega'}
a(Du_t)\cdot D(\dot u_t\psi_t)\,\dd x
+
\frac1\eps\int_{\Omega'}
W'(u_t)\dot u_t\psi_t\,\dd x
\right]\dd t
=0,
\]
where the last equality follows from the weak equation for $u_t$.
Hence
\[
\diver_{x,z}X=0
\qquad\text{in }\mathcal D'(\mathcal S).
\]

By \eqref{eq:local-parameter-control}, there exists $\eta>0$ such that
the region between the graphs of $w$ and $u_\eps$ is contained in
$\mathcal S$ whenever
\[
w-u_\eps\in W^{1,2}_0(\Omega),
\qquad
\|w-u_\eps\|_{L^\infty(\Omega)}<\eta.
\]
Using the divergence-free property and
\eqref{eq:Hilbert-excess-sec8}, we obtain
\begin{equation}\label{eq:calibration-gap}
E_\eps(w;\Omega)-E_\eps(u_\eps;\Omega)
\ge
\frac{\eps\lambda}{2}
\int_\Omega
|Dw-p(x,w(x))|^2\,\dd x.
\end{equation}
Thus $u_\eps$ is an $L^\infty$-local minimizer.

If equality holds, then
\[
Dw=p(x,w(x))
=
Du_{\vartheta(x,w(x))}(x)
\qquad\text{a.e. in }\Omega.
\]
Set $\vartheta_w(x):=\vartheta(x,w(x))$. Since
$w(x)=u_{\vartheta_w(x)}(x)$, the chain rule gives
\[
Dw
=
Du_{\vartheta_w(x)}(x)
+
\dot u_{\vartheta_w(x)}(x)D\vartheta_w
\qquad\text{a.e.}
\]
The lower bound in \eqref{eq:local-parameter-control} gives
$\dot u_t\ge c$ almost everywhere, so $D\vartheta_w=0$.
Hence $\vartheta_w$ is constant on each connected component of $\Omega$.
The zero boundary trace of $w-u_0$ and the strict ordering of the family
give $\vartheta_w=0$. Therefore $w=u_\eps$, and the minimum is strict.
\end{proof}

\begin{remark}\label{rem:stability-vs-local-minimality}
Stability is a second-variation condition, while $L^\infty$-local
minimality compares the full energy with nearby competitors. When the
second variation is defined, $L^\infty$-local minimality implies stability
by considering $u_\eps+t\varphi$ for small $t$. Stability alone does not
give $L^\infty$-local minimality. Neither condition gives global minimality. In
\cref{sec:monotone}, we prove global minimality separately.
\end{remark}

\section{Global minimization and monotone lifting}
\label{sec:monotone}

We now prove \cref{MainThm}. The key point from
\cref{sec:stability} is the ordered family of four-dimensional solutions.
After rescaling and reflecting this family across the two $\R^2$ factors,
we obtain lower and upper barriers for the scale-one equation in $\R^4$.
Following Liu--Wang--Wei \cite{LiuWangWei2017}, we minimize between these
barriers and obtain a nonconstant global minimizer.

The $L^\infty$-local minimality proved in \cref{sec:stability} is not
enough for the Jerison--Monneau construction, which requires a global
minimizer. Once the four-dimensional global minimizer is obtained, we use
the five-dimensional extension of the Mooney--Yang anisotropy and apply
the lifting argument of Jerison--Monneau \cite{JerisonMonneau2004}.

The comparison arguments needed in both steps remain valid in the
anisotropic setting. As in \cref{lem:difference-equation}, the difference
of two bounded weak solutions satisfies a uniformly elliptic equation with
bounded coefficients. We can therefore use the maximum principle and
Harnack's inequality in the barrier argument and in the vertical sliding
used in the lifting.

The symmetry is also used to show that the final solution is not
one-dimensional. The four-dimensional minimizer is
$O(2)\times O(2)$-invariant and has zero gradient at the origin, while the
lifting preserves the symmetry in the first four variables. A
one-dimensional solution with this symmetry would depend only on $x_5$;
the normalization of $\partial_5U(0)$ below rules this out.

\subsection{A four-dimensional global minimizer}

We first construct the barriers for the Liu--Wang--Wei argument. Write $H_4:=H$, $a_4:=a$, $E_{\eps,4}:=E_\eps$ and
\[
\Gamma^+:=\Gamma,
\qquad
\mathscr S(x,y):=(y,x),
\qquad
\Gamma^-:=\mathscr S\Gamma^+.
\]
For sufficiently small $\delta>0$, let $u_\delta^+:=u_\delta$ be the
solution given by \cref{thm:main}.  By the block-exchange symmetry of $F$
and the evenness of $W$,
\begin{equation}\label{eq:reflected-layer-sec9}
u_\delta^-(X):=-u_\delta^+(\mathscr SX)
\end{equation}
is the corresponding layer near $\Gamma^-$.  For $\ell$ sufficiently
large, set
\begin{equation}\label{eq:complete-barriers-sec9}
\underline u_\ell(X):=u^+_{1/\ell}(X/\ell),
\qquad
\overline u_\ell(X):=u^-_{1/\ell}(X/\ell).
\end{equation}
Since $a_4$ is one-homogeneous, both solve
\begin{equation}\label{eq:scale-one-4d-sec9}
-\diver a_4(Du)+W'(u)=0
\qquad\text{in }\R^4.
\end{equation}
By \cref{thm:ordered-family}, applied successively at comparable scales,
and the symmetry above, for all sufficiently large $\ell_2>\ell_1$,
\begin{equation}\label{eq:same-side-order-sec9}
\underline u_{\ell_2}<\underline u_{\ell_1},
\qquad
\overline u_{\ell_2}>\overline u_{\ell_1},
\end{equation}
and
\begin{equation}\label{eq:barrier-limits-sec9}
\underline u_\ell\longrightarrow-1,
\qquad
\overline u_\ell\longrightarrow1
\qquad\text{locally uniformly as }\ell\to\infty.
\end{equation}
It remains only to compare the two families.

\begin{lemma}\label{lem:cross-order-sec9}
There exists $\ell_0<\infty$ such that
\begin{equation}\label{eq:cross-order-sec9}
\underline u_\ell<\overline u_\ell
\qquad\text{in }\R^4
\end{equation}
for every $\ell\ge\ell_0$.
\end{lemma}

\begin{proof}
Let $R=|X|$.  The asymptotics of the two Mooney--Yang leaves give a
transverse separation of order $\ell^{1+\mu}(\ell+R)^{-\mu},$
whereas the scaled interface error is bounded by $C\ell^{\gamma-1}(\ell+R)^{-\gamma}
+C(\ell+R)^{-2}.$ Since $\gamma>\mu$, the ratio of the error to the separation is
$O(\ell^{-2})$, uniformly in $R$. Thus, for $\ell_0$ sufficiently large,
the two interfaces remain globally disjoint and retain the order of
$\ell\Gamma^+$ and $\ell\Gamma^-$.  The transition-layer comparison and
the maximum principle for the difference equation in
\cref{lem:difference-equation} then give \eqref{eq:cross-order-sec9}.
\end{proof}

Fix such an $\ell_0$. We now use these barriers as in
\cite[Section~2.1]{LiuWangWei2017}.

\begin{proposition}\label{prop:four-dimensional-global-minimizer}
There exists a nonconstant $O(2)\times O(2)$-invariant $E_{1,4}$ global minimizer $v\in C^{1,\alpha}_{\loc}(\R^4)$, $-1<v<1$.  Moreover,
\begin{equation}\label{eq:v-equation-sec9}
-\diver a_4(Dv)+W'(v)=0
\qquad\text{in }\R^4,
\end{equation}
and
\begin{equation}\label{eq:v-origin-sec9}
Dv(0)=0.
\end{equation}
\end{proposition}

\begin{proof}
Set
$\underline u:=\underline u_{\ell_0}$ and
$\overline u:=\overline u_{\ell_0}$.
For $R>1$, let $v_R$ minimize $E_{1,4}$ in $B_R$ with boundary value $g_R:=\frac{\underline u+\overline u}{2}.$ The direct method gives a minimizer. Truncation and the maximum principle
give $-1<v_R<1$, in $B_R$.

We claim that
\begin{equation}\label{eq:vR-trapping-sec9}
\underline u<v_R<\overline u
\qquad\text{in }B_R.
\end{equation}
By \eqref{eq:barrier-limits-sec9},
$\underline u_\ell<v_R$ for $\ell$ sufficiently large.
Let $\ell$ move to $\ell_0$. The boundary values remain strictly ordered
by \eqref{eq:same-side-order-sec9}, and an interior first contact is ruled
out by \cref{lem:difference-equation} and the strong maximum principle.
This gives $\underline u<v_R$. The upper bound is proved in the same way.

We may choose $v_R$ to be $O(2)\times O(2)$-invariant.
For $T\in O(2)\times O(2)$, the function $v_R\circ T$ is also a minimizer
with the same boundary value. The functions $\min\{v_R,v_R\circ T\}$ and $\max\{v_R,v_R\circ T\}$ are again minimizers. They are ordered and agree at the origin, so
\cref{lem:difference-equation} and the strong maximum principle give
$v_R=v_R\circ T$.

After taking a subsequence, local $C^{1,\alpha}$ estimates give us
\[
v_R\longrightarrow v
\qquad\text{in }C^1_{\loc}(\R^4).
\]
Let $\Omega\Subset\R^4$, $w-v\in W^{1,2}_0(\Omega)$ and $R$ large enough, $E_{1,4}(v_R;\Omega)
\le
E_{1,4}(v_R+w-v;\Omega).$ Passing to the limit gives $E_{1,4}(v;\Omega)\le E_{1,4}(w;\Omega).$
Thus $v$ is a global minimizer and satisfies
\eqref{eq:v-equation-sec9}.

Choose $P_+,P_-\in\R^4$ such that $\underline u(P_+)>\frac12$ and $\overline u(P_-)<-\frac12$. Passing \eqref{eq:vR-trapping-sec9} to the limit gives $v(P_+)>\frac12$ and $v(P_-)<-\frac12$. So $v$ is nonconstant. The maximum principle gives $-1<v<1$, and the
$O(2)\times O(2)$ symmetry gives \eqref{eq:v-origin-sec9}.
\end{proof}

\subsection{The anisotropic Jerison--Monneau lifting}

Now introduce the five-dimensional anisotropy used in the lifting. Mooney--Yang construct a uniformly elliptic
one-homogeneous extension $F_5:\R^5\to[0,\infty)$ of $F$, smooth away from $0$, such that
\begin{equation}\label{eq:F5-extension}
F_5(p',0)=F(p')
\qquad (p'\in\R^4).
\end{equation}
The extension has the same $O(2)\times O(2)$ symmetry in the first four
variables and is even in the fifth variable; see
\cite[Section~3]{MooneyYang2024}. Set $H_5:=\frac12F_5^2$, $a_5:=DH_5$, and for $m=4,5$,
\begin{equation}\label{eq:energy-m-sec9}
E_{\eps,m}(u;\Omega)
:=
\int_\Omega
\left[
\eps H_m(Du)+\frac1\eps W(u)
\right]\dd x.
\end{equation}
Recall that $E_{\eps,4}=E_\eps$.

We apply the Jerison--Monneau lifting to the global minimizer $v$
constructed above. The two facts needed in the anisotropic setting are
the global minimality of the two endpoint solutions and the comparison
principle used in the vertical sliding.

Since $H_5$ is convex and even in the fifth variable,
$s\mapsto H_5(p',s)$ is minimized at $s=0$. Hence
\begin{equation}\label{eq:endpoint-convexity-sec9}
H_5(p',s)\ge H_5(p',0)=H_4(p'),
\qquad
H_5(p',s)\ge H_5(0,s).
\end{equation}
For the second one, average $p'$ over the
$O(2)\times O(2)$ action. Its average is $0$, so convexity and invariance
give
\[
H_5(0,s)
\le
\int_{O(2)\times O(2)} H_5(Tp',s)\,\dd T
=
H_5(p',s).
\]

Let $q_5$ be the one-dimensional heteroclinic
\begin{equation}\label{eq:q5-sec9}
-F_5(e_5)^2q_5''+W'(q_5)=0,
\qquad
q_5(\pm\infty)=\pm1,
\qquad
q_5(0)=0.
\end{equation}
Slicing and \eqref{eq:endpoint-convexity-sec9}, together with the global
minimality of $v$ and $q_5$, show that
\[
V(x',x_5):=v(x'),
\qquad
Q(x',x_5):=q_5(x_5)
\]
are global minimizers of $E_{1,5}$.

The proof of \cref{lem:difference-equation} applies in $\R^5$ with
$H_5$ in place of $H$. Thus the difference $w=U_2-U_1$ of two bounded
weak solutions satisfies
\begin{equation}\label{eq:secant-5d}
-\diver(A_{12}Dw)+c_{12}w=0,
\qquad
\lambda_5I\le A_{12}\le\Lambda_5I,
\end{equation}
where $A_{12}$ is measurable and symmetric and $c_{12}$ is bounded.
Hence the maximum principle and Harnack's inequality used in the
Jerison--Monneau sliding argument remain available.

\begin{proposition}\label{prop:anisotropic-JM}
Let $v$ be the global minimizer obtained in
\cref{prop:four-dimensional-global-minimizer}, set $s_0:=v(0)$, and define
\begin{equation}\label{eq:gamma-star-sec9}
\kappa_*
:=
q_5'\bigl(q_5^{-1}(s_0)\bigr)
=
\frac{\sqrt{2W(s_0)}}{F_5(e_5)}.
\end{equation}
For every $\kappa\in(0,\kappa_*)$, there exists a global minimizer
\[
U\in C^\infty(\R^5),
\qquad
-1<U<1,
\]
such that
\begin{equation}\label{eq:monotone-scale-one-sec9}
-\diver a_5(DU)+W'(U)=0,
\qquad
\partial_5U>0
\qquad\text{in }\R^5,
\end{equation}
and
\begin{equation}\label{eq:JM-normalization-sec9}
U(0)=s_0,
\qquad
\partial_5U(0)=\kappa.
\end{equation}
The solution is even in the first four variables and is not
one-dimensional.
\end{proposition}

\begin{proof}
Consider the Jerison--Monneau family
\begin{equation}\label{eq:JM-homotopy-sec9}
B_t(x',x_5)
:=
q_5\bigl((1-t)x_5+tq_5^{-1}(v(x'))\bigr),
\qquad
0\le t\le1.
\end{equation}
We have $B_0=Q$, $B_1=V$, and $B_t$ is strictly increasing in $x_5$
for $t<1$.

Following \cite{JerisonMonneau2004}, minimize $E_{1,5}$ on finite
cylinders with boundary values obtained from $B_t$. The direct method and
truncation give minimizers with values in $[-1,1]$. Sliding in the
$x_5$ direction and using \eqref{eq:secant-5d} gives monotonicity in
$x_5$.

Letting the cylinder height tend to infinity gives minimizing solutions on
$B_R^4\times\R$. At $t=0$ and $t=1$ the limits are $Q$ and $V$,
respectively. At $t=1$, uniqueness follows from global minimality.
Indeed, if two minimizers have the same boundary values, their pointwise
minimum and maximum are again minimizers. They are ordered, and
\eqref{eq:secant-5d} together with the strong maximum principle shows
that they agree.

For $t<1$, translate the solution in the $x_5$ direction so that its value
at the origin is $s_0$. Let $m_R(t)$ denote its vertical derivative at
the origin. The centering argument of Jerison--Monneau gives
\[
m_R(0)=\kappa_*,
\qquad
\lim_{t\uparrow1}m_R(t)=0,
\]
and $m_R$ is continuous. Hence, for every
$\kappa\in(0,\kappa_*)$, there exists $t_R\in(0,1)$ such that the
centered minimizer $U_R$ satisfies
\begin{equation}\label{eq:UR-normalization-sec9}
U_R(0)=s_0,
\qquad
\partial_5U_R(0)=\kappa,
\qquad
\partial_5U_R\ge0.
\end{equation}

After taking a subsequence as $R\to\infty$, local compactness gives a
global minimizer $U$ satisfying
\eqref{eq:JM-normalization-sec9} and $\partial_5U\ge0$.
For $h>0$, set
\[
Z_h(x):=\frac{U(x+he_5)-U(x)}{h}\ge0.
\]
By the same difference equation,
$Z_h$ satisfies a uniformly elliptic equation of the form
\eqref{eq:secant-5d}. Since $Z_h(0)\to\kappa>0$, Harnack's inequality
gives a positive lower bound on every compact set. Letting $h\downarrow0$
gives
\begin{equation}\label{eq:strict-monotonicity-sec9}
\partial_5U>0
\qquad\text{in }\R^5.
\end{equation}

Thus $DU\neq0$ everywhere. Since $H_5$ is smooth away from $0$, elliptic
regularity gives $U\in C^\infty(\R^5)$, and the maximum principle gives
$-1<U<1$.

The construction preserves evenness in the first four variables. If $U$
were one-dimensional, this symmetry and
\eqref{eq:strict-monotonicity-sec9} would imply
\[
U(x)=q_5(x_5+c)
\]
for some $c\in\R$. Since $U(0)=s_0$,
$c=q_5^{-1}(s_0)$, and hence
\[
\partial_5U(0)
=
q_5'\bigl(q_5^{-1}(s_0)\bigr)
=
\kappa_*.
\]
This contradicts
$\partial_5U(0)=\kappa<\kappa_*$. Therefore $U$ is not
one-dimensional.
\end{proof}

For $\eps>0$, set $U_\eps(x):=U(x/\eps)$. Since $a_5$ is
one-homogeneous,
\[
-\eps^2\diver a_5(DU_\eps)+W'(U_\eps)=0,
\qquad
\partial_5U_\eps>0.
\]
The energies satisfy
\[
E_{\eps,5}(U_\eps;\Omega)
=
\eps^4E_{1,5}(U;\Omega/\eps),
\]
so $U_\eps$ is again a global minimizer. It is smooth, strictly
monotone, and not one-dimensional. This proves \cref{MainThm}.

\begin{remark}
The five-dimensional solution is obtained from the ordered family in
$\R^4$: the ordering gives the barriers for the global minimization of
Liu--Wang--Wei \cite{LiuWangWei2017}, and the resulting global minimizer
is then used in the Jerison--Monneau lifting
\cite{JerisonMonneau2004}.

This is different from the construction of del Pino--Kowalczyk--Wei \cite{delPinoKowalczykWei2011}, where the transition layer is glued directly near a large dilation of the Bombieri--De Giorgi--Giusti minimal graph
\cite{BombieriDeGiorgiGiusti1969}. Mooney--Yang also construct an entire anisotropic minimal graph in
$\R^5$, which is its anisotropic counterpart. However, our construction of the monotone solution in
$\R^5$ does not show that its nodal set is close to a large dilation of
this Mooney--Yang graph.
\end{remark}

\nosection{Acknowledgements}

The research of J.Wei is supported by National Natural Science Foundation of China (No. 12631008) and is partially supported by GRF of RGC of Hong Kong entitled “On critical and supercritical Fujita equation”. The research of Y. Wu is supported by National Natural Science Foundation of China (No. 12631008, No. 12671144), Yunnan Provincial Innovation Team on the Interdisciplinary Integration of Modern Applied Mathematics and Life Sciences (No. 202405AS350003), Yunnan Fundamental Research Projects (No. 202601CJ070001) and Yunnan Revitalization Talent Support Program. The authors used AI models to assist with calculations and writings; the main ideas, mathematical validation, and all final checks remain the sole responsibility of the human authors.

\ \\

\appendix
\section{The full weighted Jacobi inverse}
\label{app:Jacobi}

We prove \cref{thm:full-Jacobi-inverse}. The argument directly follows \cite[Section~10]{PacardWei2013} and \cite{PacardNotes}, which is the weighted Fredholm scheme based on Sobolev space and a standard barrier argument. The point here is that, for Mooney-Yang's construction, \cref{prop:full-Jacobi-Gamma} identifies the exact anisotropic model and the indicial roots is fine.

On the end of $\Gamma$, write $r=e^t$. After multiplication by $e^{2t}$, the operator in \eqref{eq:full-Jacobi-Gamma} converges exponentially to the translation-invariant
operator
\begin{equation}
  L_\infty
  :=\frac{\sqrt{2}}{2}
  \left[4\varphi''(1)(\partial_t^2+\partial_t)+\Delta_\Lambda+2\right].
  \label{eq:appendix-model-operator}
\end{equation}
Its separated solutions are $e^{\beta_j^\pm t}\phi_j$, where
$\beta_j^\pm$ are given by \eqref{eq:full-indicial-roots}.

We first record the Fredholm statement, including the precise dual weight.

\begin{lemma}[Fredholm property and dual weight]
\label{lem:appendix-fredholm-dual-weight}
Assume that
\[
  -\gamma\neq\beta_j^\pm
\]
for every $j$. Then
\begin{equation}
  J_{F,\Gamma}:C_\gamma^{2,\alpha}(\Gamma)
  \longrightarrow C_{\gamma+2}^{0,\alpha}(\Gamma)
  \label{eq:appendix-holder-map}
\end{equation}
is Fredholm. If, in addition,
\[
  -(1-\gamma)=\gamma-1\neq\beta_j^\pm
\]
for every $j$, then the unweighted pairing
\[
  (f,v)\longmapsto\int_\Gamma fv\,d\mathcal H^3
\]
induces a canonical isomorphism
\begin{equation}
  \left(
    \operatorname{coker}
    \bigl(J_{F,\Gamma}\big|_{C_\gamma^{2,\alpha}}\bigr)
  \right)^*
  \simeq
  \ker\bigl(J_{F,\Gamma}\big|_{C_{1-\gamma}^{2,\alpha}}\bigr).
  \label{eq:appendix-cokernel-duality}
\end{equation}
Moreover, there are a compact set $K\Subset\Gamma$ and a constant $C$ such
that
\begin{equation}
  \|u\|_{C_\gamma^{2,\alpha}(\Gamma)}
  \leq C\left(
    \|J_{F,\Gamma}u\|_{C_{\gamma+2}^{0,\alpha}(\Gamma)}
    +\|u\|_{C^0(K)}
  \right).
  \label{eq:appendix-holder-apriori}
\end{equation}
\end{lemma}

\begin{proof}
We give the Fredholm and duality argument in some detail because the exact Hölder weights in \eqref{eq:appendix-cokernel-duality} are borderline for the
unweighted $L^2(\Gamma)$ pairing.

Choose a smooth positive function $a$ on $\Gamma$ such that
\[
  a(t,q)=e^{2t}
\]
for all sufficiently large $t$, and set
\begin{equation}
  \widehat J:=aJ_{F,\Gamma},
  \qquad
  d\nu:=a^{-1}\,d\mathcal H^3.
  \label{eq:appendix-rescaled-operator}
\end{equation}
Multiplication by $a$ is an isomorphism from
$C_{\gamma+2}^{0,\alpha}(\Gamma)$ onto
$C_\gamma^{0,\alpha}(\Gamma)$. By (3.23), on the end we have
\begin{equation}
  \widehat J=L_\infty+\mathcal R,
  \label{eq:appendix-operator-perturbation}
\end{equation}
where the coefficients of the second-order operator $\mathcal R$, together
with the derivatives needed in the estimates below, converge exponentially
to zero.

The rescaled operator $\widehat J$ is formally self-adjoint with respect to
$d\nu$. Indeed, for compactly supported $u,v$,
\begin{align}
  \int_\Gamma u\widehat Jv\,d\nu
  &=\int_\Gamma uJ_{F,\Gamma}v\,d\mathcal H^3 \notag\\
  &=\int_\Gamma vJ_{F,\Gamma}u\,d\mathcal H^3
   =\int_\Gamma v\widehat Ju\,d\nu.
  \label{eq:appendix-formal-self-adjointness}
\end{align}
In the coordinates $(t,q)$ on the end,
\begin{equation}
  d\nu=b(t,q)e^t\,dt\,d\mathcal H_\Lambda^2,
  \label{eq:appendix-rescaled-measure}
\end{equation}
where $b$ is positive, is bounded above and below uniformly, and converges
exponentially to a positive limit.

For $s\in\mathbb R$, let $\mathcal H_s^k(\Gamma)$ be the weighted cylindrical
Sobolev space with norm equivalent on the end to
\begin{equation}
  \sum_{\ell=0}^k
  \left\|e^{st}\nabla_{t,q}^{\ell}u\right\|_{L^2(dt\,d\mathcal H_\Lambda^2)},
  \label{eq:appendix-sobolev-norm}
\end{equation}
together with the usual $H^k$-norm on a fixed compact subset of $\Gamma$.
The critical values of $s$ for
\begin{equation}
  \widehat J:\mathcal H_s^2(\Gamma)
  \longrightarrow\mathcal H_s^0(\Gamma)
  \label{eq:appendix-sobolev-map}
\end{equation}
are precisely the values satisfying $-s=\beta_j^\pm$ for some $j$.

We use the following standard weight-improvement consequence of the indicial
analysis. Suppose that a closed interval of decay weights contains no number
$-\beta_j^\pm$. After decreasing $\eta>0$ if necessary, one has
\begin{align}
  u\in\mathcal H_{\lambda-\eta}^2,
  \quad \widehat Ju\in C_\lambda^{0,\alpha}
  &\quad\Longrightarrow\quad
  u\in C_\lambda^{2,\alpha},
  \label{eq:appendix-inhomogeneous-weight-improvement}\\
  \widehat Ju=0,
  \quad u\in C_\lambda^{2,\alpha}
  &\quad\Longrightarrow\quad
  u\in C_{\lambda+\eta}^{2,\alpha}.
  \label{eq:appendix-homogeneous-weight-improvement}
\end{align}
For completeness, these implications can be seen by cutting off the solution
on the end and expanding it in the eigenfunctions $\phi_j$ of
$-(\Delta_\Lambda+2)$. Variation of parameters for each ordinary differential
equation associated with $L_\infty$ shows that a change of weight can fail only
when one crosses an exponent $-\beta_j^\pm$. The perturbation $\mathcal R$ in
\eqref{eq:appendix-operator-perturbation} is absorbed after moving the initial
point of the end sufficiently far out. Interior and dyadic Schauder estimates
then give
\eqref{eq:appendix-inhomogeneous-weight-improvement} and
\eqref{eq:appendix-homogeneous-weight-improvement}.

For every noncritical $s$, the same model inverse and perturbation argument
gives
\begin{equation}
  \|u\|_{\mathcal H_s^2}
  \leq C\left(
    \|\widehat Ju\|_{\mathcal H_s^0}+\|u\|_{L^2(K)}
  \right)
  \label{eq:appendix-sobolev-apriori}
\end{equation}
for some fixed compact set $K\Subset\Gamma$. It follows that the kernel of
\eqref{eq:appendix-sobolev-map} is finite dimensional.

The range is closed. Otherwise, after restricting to a closed complement of
the kernel, there would exist a sequence $u_i$ such that
\[
  \|u_i\|_{\mathcal H_s^2}=1,
  \qquad
  \|\widehat Ju_i\|_{\mathcal H_s^0}\longrightarrow0.
\]
After passing to a subsequence, weak compactness and Rellich compactness on
$K$ give $u_i\rightharpoonup u$ in $\mathcal H_s^2$ and
$u_i\to u$ in $L^2(K)$. The limit belongs to the kernel and to the chosen
complement, so $u=0$. Estimate \eqref{eq:appendix-sobolev-apriori} then gives
$\|u_i\|_{\mathcal H_s^2}\to0$, a contradiction.

The pairing induced by $d\nu$ identifies the dual of $\mathcal H_s^0$ with
$\mathcal H_{1-s}^0$. Indeed, by \eqref{eq:appendix-rescaled-measure},
\begin{equation}
  \left|\int_\Gamma Fv\,d\nu\right|
  \leq C
  \|e^{st}F\|_{L^2(dt\,dq)}
  \|e^{(1-s)t}v\|_{L^2(dt\,dq)}.
  \label{eq:appendix-sobolev-dual-pairing}
\end{equation}
Formal self-adjointness and elliptic regularity therefore identify the
annihilator of the range of \eqref{eq:appendix-sobolev-map} with
\begin{equation}
  \ker\left(
    \widehat J:\mathcal H_{1-s}^2
    \longrightarrow\mathcal H_{1-s}^0
  \right).
  \label{eq:appendix-sobolev-cokernel}
\end{equation}
Here formal self-adjointness first applies to compactly supported functions;
the identity extends to the weighted spaces by cutoff and density, using
\eqref{eq:appendix-sobolev-dual-pairing}. Choosing $s$ so that both $s$ and
$1-s$ are noncritical, estimate \eqref{eq:appendix-sobolev-apriori} at the
complementary weight shows that the space in
\eqref{eq:appendix-sobolev-cokernel} is finite dimensional. Thus
\eqref{eq:appendix-sobolev-map} is Fredholm.

We now pass to the exact Hölder weights. Fix $\gamma$ with
$-\gamma\neq\beta_j^\pm$ for every $j$. Choose $\eta>0$ sufficiently small
that the interval $[\gamma-2\eta,\gamma+2\eta]$ contains no indicial decay
rate, and so that both $s:=\gamma-\eta$ and $1-s$ are noncritical. If
$f\in C_{\gamma+2}^{0,\alpha}(\Gamma)$ and $F:=af$, then
\begin{equation}
  F\in C_\gamma^{0,\alpha}(\Gamma)
  \subset\mathcal H_{\gamma-\eta}^0(\Gamma).
  \label{eq:appendix-holder-to-sobolev}
\end{equation}
The Sobolev Fredholm alternative shows that
\begin{equation}
  \widehat Ju=F,
  \qquad
  u\in\mathcal H_{\gamma-\eta}^2,
  \label{eq:appendix-sobolev-equation}
\end{equation}
is solvable if and only if
\begin{equation}
  \int_\Gamma Fv\,d\nu=0
  \label{eq:appendix-solvability-condition}
\end{equation}
for every
\begin{equation}
  v\in K_\gamma^*
  :=\ker\left(
    \widehat J:\mathcal H_{1-\gamma+\eta}^2
    \longrightarrow\mathcal H_{1-\gamma+\eta}^0
  \right).
  \label{eq:appendix-dual-kernel}
\end{equation}
Whenever a Sobolev solution exists,
\eqref{eq:appendix-inhomogeneous-weight-improvement}, applied between the
weights $\gamma-\eta$ and $\gamma$, gives
\begin{equation}
  u\in C_\gamma^{2,\alpha}(\Gamma).
  \label{eq:appendix-sobolev-to-holder}
\end{equation}

Elements of $K_\gamma^*$ decay strictly faster than
$\rho^{-(1-\gamma)}$. Hence each $v\in K_\gamma^*$ defines a bounded functional
on $C_{\gamma+2}^{0,\alpha}$ by
\begin{equation}
  \ell_v(f)
  :=\int_\Gamma fv\,d\mathcal H^3
  =\int_\Gamma Fv\,d\nu.
  \label{eq:appendix-cokernel-functional}
\end{equation}
If $f=J_{F,\Gamma}u$ for $u\in C_\gamma^{2,\alpha}$, integration by parts on
$\Gamma\cap B_R$ gives $\ell_v(f)=0$. The boundary term tends to zero as
$R\to\infty$: if $v=O(\rho^{-(1-\gamma+\eta')})$ for some $\eta'>0$, its size
is $O(R^{-\eta'})$. Conversely, if $\ell_v(f)=0$ for every
$v\in K_\gamma^*$, the Sobolev Fredholm alternative produces a solution of
\eqref{eq:appendix-sobolev-equation}, and
\eqref{eq:appendix-sobolev-to-holder} upgrades it to
$C_\gamma^{2,\alpha}$. Consequently,
\begin{equation}
  \operatorname{Ran}\left(
    J_{F,\Gamma}\big|_{C_\gamma^{2,\alpha}}
  \right)
  =\bigcap_{v\in K_\gamma^*}\ker\ell_v.
  \label{eq:appendix-holder-range}
\end{equation}
This range is closed and has finite codimension. Since the kernel in the
Hölder space is contained in the finite-dimensional Sobolev kernel, the map
in \eqref{eq:appendix-holder-map} is Fredholm. The end Schauder estimate used
above, together with an interior estimate on $K$, also gives
\eqref{eq:appendix-holder-apriori}.

It remains to identify the exact complementary weight. Assume that
$1-\gamma$ is also noncritical, and reduce $\eta$ so that a neighborhood of
$1-\gamma$ contains no indicial decay rate. Weighted elliptic regularity and
\eqref{eq:appendix-homogeneous-weight-improvement} give
\begin{equation}
  K_\gamma^*
  =\ker\left(
    J_{F,\Gamma}:C_{1-\gamma}^{2,\alpha}
    \longrightarrow C_{3-\gamma}^{0,\alpha}
  \right).
  \label{eq:appendix-exact-dual-kernel}
\end{equation}
Indeed, an element of $K_\gamma^*$ has decay strictly faster than
$\rho^{-(1-\gamma)}$ and therefore lies in
$C_{1-\gamma}^{2,\alpha}$. Conversely, a Jacobi field in
$C_{1-\gamma}^{2,\alpha}$ improves to
$C_{1-\gamma+\eta'}^{2,\alpha}$ for some $\eta'>\eta$, after decreasing
$\eta$ if necessary, and hence belongs to
$\mathcal H_{1-\gamma+\eta}^2$.

Equations \eqref{eq:appendix-holder-range} and
\eqref{eq:appendix-exact-dual-kernel} show that
$v\mapsto\ell_v$ induces the isomorphism
\eqref{eq:appendix-cokernel-duality}.
\end{proof}

The dilation foliation gives the injectivity needed at both weights. Define
\begin{equation}
  \zeta_0(X):=X\cdot\nu_\Gamma(X),
  \qquad X\in\Gamma,
  \label{eq:appendix-dilation-field}
\end{equation}
with the orientation chosen so that $\zeta_0>0$. Since dilations preserve
$F$-minimality,
\begin{equation}
  J_{F,\Gamma}\zeta_0=0.
  \label{eq:appendix-dilation-jacobi}
\end{equation}
Moreover, by the decay of $\Gamma$,
\begin{equation}
  \zeta_0(X)\simeq\rho(X)^{-\mu}
  \qquad\text{as }|X|\longrightarrow\infty.
  \label{eq:appendix-dilation-asymptotics}
\end{equation}

\begin{lemma}[Injectivity beyond the slow mode]
\label{lem:appendix-injectivity}
For every $\gamma>\mu$,
\begin{equation}
  \ker\left(
    J_{F,\Gamma}:C_\gamma^{2,\alpha}(\Gamma)
    \longrightarrow C_{\gamma+2}^{0,\alpha}(\Gamma)
  \right)=\{0\}.
  \label{eq:appendix-injective-map}
\end{equation}
\end{lemma}

\begin{proof}
The function $\zeta_0(X)$ of Jacobi field solving $$\mathrm{div}_\Gamma[\Psi(\nu)(\nabla\zeta_0)]+\mathrm{tr}_\Gamma (\Psi(\nu)A_\Gamma^2) \zeta_0= 0.$$ 
By the property of Mooney-Yang's foliation, $\zeta_0\neq 0$ everywhere and does not change sign.

Then according to \cite[Lemma 2.1]{MooneyYang2021}, $\zeta_0$ does not belong to $C^{2,\alpha}_\gamma(\Gamma)$ for $\gamma >\mu$. Since this Jacobi field does not change sign, it can be used as a barrier to prove injectivity in the corresponding spaces. 

Here we state the barrier argument in detail: let $u\in C^{2,\alpha}_\gamma(\Gamma),J_\Gamma u=0, \gamma>\mu$. Denote
$$J_\Gamma u = L(u)+Vu,\, L(u)\coloneqq \operatorname{div}_\Gamma[\Psi(\nu)(\nabla u)],\, V\coloneqq \operatorname{tr}_\Gamma(\Psi(\nu)A_\Gamma^2).$$
Let \(w = u/\zeta_0\), a.e. \(u = w\zeta_0\). By
$$\begin{aligned}
L(w\zeta_0) &= \operatorname{div}_\Gamma\big[\Psi(\nu)\nabla(w\zeta_0)\big] \\
&= \operatorname{div}_\Gamma\big[\zeta_0 \Psi(\nu)\nabla w\big] + \operatorname{div}_\Gamma\big[w \Psi(\nu)\nabla \zeta_0\big]\\
&=\zeta_0\,L(w) + w\,L(\zeta_0) + \langle \nabla \zeta_0, \Psi(\nu)\nabla w \rangle + \langle \nabla w, \Psi(\nu)\nabla \zeta_0 \rangle\\
&=\zeta_0\,L(w) + w\,L(\zeta_0) + 2\langle \nabla \zeta_0, \Psi(\nu)\nabla w \rangle
\end{aligned}$$
we have
$$\begin{aligned}
0 = J_\Gamma(w\zeta_0) &= L(w\zeta_0) + Vw\zeta_0 = \zeta_0 L(w) + 2\langle \nabla \zeta_0, \Psi(\nu)\nabla w \rangle,
\end{aligned}$$
hence we get the equation of $w$
$$\begin{aligned}
\operatorname{div}_\Gamma[\zeta_0^2\,\Psi(\nu)\nabla w]
&= \zeta_0^2\,L(w) + \langle \nabla(\zeta_0^2), \Psi(\nu)\nabla w \rangle \\
&= \zeta_0(\zeta_0 L(w) + 2\langle \nabla \zeta_0, \Psi(\nu)\nabla w\rangle) = 0.
\end{aligned}$$
By definition, $u$ decay faster than $\zeta_0$, so $w=u/\zeta_0$ vanish at infinity. By maximum principle, $w\equiv 0$, so $u\equiv 0$. 

\end{proof}

\begin{proof}[Proof of Theorem \ref{thm:full-Jacobi-inverse}]
First, by Proposition \ref{prop:cone-Jacobi}, every
\[
  \gamma\in(\mu,1-\mu)
\]
avoids the indicial roots. $1-\gamma$ belongs to the
same interval and therefore also avoids them.

Lemma \ref{lem:appendix-fredholm-dual-weight} shows that
\[
  J_{F,\Gamma}:C_\gamma^{2,\alpha}
  \longrightarrow C_{\gamma+2}^{0,\alpha}
\]
is Fredholm. Since $\gamma>\mu$, Lemma
\ref{lem:appendix-injectivity} gives injectivity. By
\eqref{eq:appendix-cokernel-duality}, the dual of its cokernel is the kernel at
weight $1-\gamma$. Since $1-\gamma>\mu$, Lemma
\ref{lem:appendix-injectivity} shows that this kernel is also zero. The map is
therefore surjective and hence is an isomorphism.

For $k=0$, estimate \eqref{eq:full-Jacobi-estimate} follows from
\eqref{eq:appendix-holder-apriori} and injectivity by the usual contradiction
argument. For $k\geq1$, both the estimate and the isomorphism statement follow
from weighted Schauder regularity on the dyadic charts. This proves \eqref{eq:full-Jacobi-isomorphism} and
\eqref{eq:full-Jacobi-estimate}.
\end{proof}

\bibliographystyle{amsalpha}
\bibliography{references}
 
\end{document}